\RequirePackage{fix-cm}
\documentclass[smallcondensed,numbook,envcountsect]{svjour3}     % onecolumn (ditto)
\smartqed  % flush right qed marks, e.g. at end of proof
\usepackage{graphicx}
\usepackage{amssymb,amsmath,amsfonts}
\usepackage{setspace} 
\usepackage{hyperref}
\usepackage{color}
\usepackage{fancyhdr}
\usepackage{booktabs}
\usepackage{tabularx}
\usepackage{enumerate}
\usepackage{enumitem}
\usepackage{wrapfig}
\usepackage{mdframed}
\usepackage{lmodern}
\usepackage[notref,notcite,final]{showkeys}
\usepackage[final]{pdfpages}
\usepackage{tikz}
\usetikzlibrary{positioning,arrows,shapes,snakes}
\tikzset{main node/.style={circle,draw,minimum size=0.3em,inner
sep=0.5pt}}
\tikzset{state node/.style={circle,draw,minimum size=0.5em,fill=blue!20,inner sep=0.5pt}}
\tikzset{small node/.style={circle,draw,minimum size=0.5em,inner
sep=2pt,font=\sffamily\bfseries}}
\usepackage{cancel}
\usetikzlibrary{matrix, arrows, decorations.pathmorphing}
\usepackage[retainorgcmds]{IEEEtrantools}

\newcommand{\N}{\mathbb{N}}

\newcommand{\Z}{\mathbb{Z}}

\newcommand{\cala}{\mathcal{A}}

\newcommand{\cald}{\mathcal{D}}
\newcommand{\calh}{{H}}

\newcommand{\la}{\mathbf{a}}

\newcommand{\ra}{\rightarrow}

\newcommand{\se}{\subseteq}

\newcommand{\ip}[1]{\langle#1\rangle}

\newcommand{\inverse}{^{-1}}

\newcommand{\abs}[1]{\lvert #1 \rvert}

\spnewtheorem{notation}{Notation}[section]{\bf}{\rm}
\spnewtheorem{assumption}{Assumption}[section]{\bf}{\rm}
\spnewtheorem*{thm1.4}{Theorem 1.4}{\bf}{\it}
\spnewtheorem*{thm1.1}{Theorem 1.1}{\bf}{\it}
\spnewtheorem*{thm1.2}{Theorem 1.2}{\bf}{\it}
\spnewtheorem*{thm1.3}{Theorem 1.3}{\bf}{\it}
\spnewtheorem*{proof1.1}{Proof of Theorem 1.1}{\it}{\rm}
\spnewtheorem*{proof1.2}{Proof of Theorem 1.2}{\it}{\rm}
\spnewtheorem*{proof1.3}{Proof of Theorem 1.3}{\it}{\rm}
\spnewtheorem*{proof1.4}{Proof of Theorem 1.4}{\it}{\rm}
\spnewtheorem*{proof6.1}{Proof of Theorem 6.1}{\it}{\rm}
\spnewtheorem*{proof6.2}{Proof of Theorem 6.2}{\it}{\rm}

\begin{document}

%\title{Insert your title here%\thanks{Grants or other notes
%%about the article that should go on the front page should be
%%placed here. General acknowledgments should be placed at the end of the article.}
%}

\title{On the subregular $J$-rings of Coxeter systems}


\author{Tianyuan Xu}
%\author{First Author         \and
        %Second Author %etc.
%}


\institute{Tianyuan Xu \at
               Department of Mathematics and Statistics\\
               Queen's University,
Kingston, Ontario, Canada, K7L 3N6\\
              %Tel.: +1-613-533-2401\\
              \email{tx7@queensu.ca}           %  \\
           %\and
           %S. Author \at
              %second address
}

\date{}

\maketitle

\begin{abstract}
%Insert your abstract here. Include keywords, PACS and mathematical
%subject classification numbers as needed.

We recall Lusztig's construction of the asymptotic Hecke algebra $J$ of a
Coxeter system $(W,S)$ via the Kazhdan--Lusztig basis of the corresponding
Hecke algebra. The algebra $J$ has a direct summand $J_E$ for each two-sided
Kazhdan--Lusztig cell of $W$, and we study the summand $J_C$
corresponding to a particular cell $C$ called the subregular cell. We develop
a combinatorial method involving truncated Clebsch--Gordan rules to compute $J_C$ without using the Kazhdan--Lusztig
basis. As applications, we deduce some connections between $J_C$ and the
Coxeter diagram of $W$, and we show that for certain Coxeter systems $J_C$
contains subalgebras that are free fusion rings in the sense of  \cite{Banica},
  thereby connecting the subalgebras to {compact quantum groups} arising from operator algebra theory.
\end{abstract}

%\section{Introduction}
%\label{intro}
%Your text comes here. Separate text sections with
%\section{Section title}
%\label{sec:1}
%Text with citations \cite{RefB} and \cite{RefJ}.
%\subsection{Subsection title}
%\label{sec:2}
%as required. Don't forget to give each section
%and subsection a unique label (see Sect.~\ref{sec:1}).
%\paragraph{Paragraph headings} Use paragraph headings as needed.
%\begin{equation}
%\end{equation}

%% For one-column wide figures use
%\begin{figure}
%% Use the relevant command to insert your figure file.
%% For example, with the graphicx package use
  %\includegraphics{example.eps}
%\end{figure}
%%
%% For two-column wide figures use
%\begin{figure*}
%% Use the relevant command to insert your figure file.
%% For example, with the graphicx package use
  %\includegraphics[width=0.75\textwidth]{example.eps}
%\end{figure*}
%\begin{table}
%\begin{tabular}{lll}
%\hline\noalign{\smallskip}
%\end{tabular}
%\end{table}

  \section{Introduction}
Hecke algebras of Coxeter systems are classical objects of study in
representation theory because of their rich connections with finite groups of
Lie type, Lie algebras, quantum groups, and the geometry of flag varieties
(see, for example, \cite{Curtis}, \cite{CurtisIwahori},
\cite{DipperJames}, \cite{GeckPfeiffer}, \cite{KL}, \cite{LusztigCharacters}). 
Let $(W,S)$ be a Coxeter system, and let $H$ be its Hecke algebra defined
over the ring $\Z[v,v\inverse]$. Using the Kazhdan--Lusztig basis of $H$, Lusztig
constructed the \emph{asymptotic Hecke algebra}
$J$ of $(W,S)$ in \cite{L2}. The algebra $J$ can be
viewed as a limit of $H$ as the parameter $v$ goes to infinity, and its
representation theory is closely related to that of $H$ (see
\cite{L2}, \cite{L3}, \cite{L4}, \cite{LG}, \cite{Geck_asymptotic}). In particular, upon
suitable extensions of scalars, $J$ admits a natural homomorphism from
$H$, hence representations of $J$ induce representations of $H$
(see \cite{LG}).

The algebra $J$ has several interesting features. First, $J$ is defined to be the free abelian group 
$J=\oplus_{w\in W}\Z t_w$, with multiplication of the basis elements
given by 
\[
  t_xt_y=\sum_{z\in W}\gamma_{x,y,z\inverse}t_z
\]
where the coefficients $\gamma_{x,y,z\inverse}$ are
nonnegative integers extracted from the structure constants
of the {Kazhdan--Lusztig basis} of $H$. The non-negativity of the structure constants
makes $J$ a \emph{$\Z_+$-ring}, and the basis elements satisfy additional
conditions which make $J$ a \emph{based ring} in the sense of
\cite{Lusztigbased} and \cite{EGNO} (see Section \ref{sec:based rings}). 

  %Beyond the based ring
  %structure, however, the nature of $J$ is largely unknown except when the
  %Coxeter group $W$ is a Weyl group or affine Weyl group. For an arbitrary
  %Coxeter system, the main obstacle in understanding $J$ lies in the difficulty
  %of computing the Kazhdan--Lusztig polynomials in the Hecke algebra. 

Another interesting feature of $J$ is that for any
two-sided Kazhdan--Lusztig cell
$E$ of $W$, the subgroup 
\[
  J_E=\oplus_{w\in E}\Z t_w
\] 
of $J$ is a subalgebra of $J$ and also a based ring. Here, each Kazhdan--Lusztig cell is a subset of $W$. The cells of
$W$ are defined using the Kazhdan--Lusztig basis of $H$ and form a partition of
$W$. Further, $J_E$ is in fact
a direct summand of $J$, and $J$ admits the direct sum decomposition 
\[
  J=\oplus_{E\in \mathcal{C}}J_E,
\]
where $\mathcal{C}$ denotes the collection of all two-sided cells of $W$
(see Proposition \ref{subalgebra}). Thus, it is natural to study $J$ by first
studying the algebras of the form $J_E$.

In this paper, we
focus on a particular two-sided cell $C$ of $W$ known as the
\emph{subregular cell}. We call its corresponding based ring $J_C$ the
\emph{subregular $J$-ring}; this is the ring referred to in the title of the
paper. We study $J_C$ and  subalgebras $J_s$ of $J_C$ that correspond to the generators $s\in S$
of $W$. Thanks to a result of Lusztig in \cite{subregular}, the cell $C$ can be
characterized as the set of non-identity elements in $W$ with
unique reduced words. The main theme of the paper is to exploit
this combinatorial characterization and study $J_C$ and $J_s (s\in S)$
without reference to Kazhdan--Lusztig bases. This is desirable since a main obstacle in
understanding $J$ for arbitrary Coxeter systems lies in the difficulty of
understanding Kazhdan--Lusztig bases. 

\begin{remark}
  \label{unit remark}
  It is worth mentioning that the algebra $J$, as well as the subalgebra $J_E$
  where
  $E$ is an arbitrary two-sided cell of $W$, do not
  generally have units in the usual sense. More specifically, we need to
  consider
  the set $\cald$ of \emph{distinguished involutions} of $W$ (see Equation
  \ref{eq:cald}) and $J$ has a unit element, namely $\sum_{d\in \cald}t_d$,
  only when $\cald$ is finite. On the other hand, when $\cald$ is infinite, the
  set $\{t_d:d\in \cald\}$ may be viewed as a generalized unit element of $J$
  in the sense that $t_dt_{d'}=\delta_{d,d'}$ for any $d,d'\in \cald$ and
  $\sum_{d,d'\in \cald}t_dJt_d'=J$ (see \cite{LG}, Section 18.3).  Similarly,
  $J_E$ has unit $\sum_{d\in E\cap \cald} t_d$ if the set $E\cap \cald$ is
  finite while otherwise the set $\{t_d:d\in E\cap \cald\}$ may be viewed as a
  generalized unit of $J_E$. For the subregular cell $C$, the set $C\cap \cald$
  turns out to be the generating set $S$ of $W$, therefore $J_C$ is unital (as
  we will assume $S$ is finite). For each $s\in S$, the algebra $J_s$ mentioned
  above will be unital as well, with the element $t_s$ as its unit (see Proposition \ref{subalgebra}
  and Remark
\ref{C distinguished}).
\end{remark}

A third important feature of the algebra $J$ is that it admits a very interesting 
\emph{categorification}. Here by categorification we mean the process of
adding an extra layer of structure to an algebraic object to produce an
interesting category which allows one to recover the object; more specially,
we mean that $J$ appears as the Grothendieck ring of a {tensor category}
$\mathcal{J}$ (see \cite{EGNO} for the definition of a tensor category, \cite{LG} for the construction of $\mathcal{J}$). A well-known example of categorification is the
categorification of the Hecke algebra $H$ by the {Soergel category}
$\mathcal{SB}$, which was used to prove the ``positivity properties'' of the
{Kazhdan--Lusztig basis} of $H$ in \cite{EW}. 

Just as the algebra $J$ is constructed from $H$, the category
$\mathcal{J}$ is constructed from $\mathcal{SB}$, also by
Lusztig (\cite{LG}). Further, just as the algebra $J$ has a subalgebra of the
form $J_E$ for each two-sided
cell $E$ and a subalgebra $J_s$ for each generator $s\in S$, the category
$\mathcal{J}$  has a subcategory $\mathcal{J}_E$ for each two-sided cell $E$
and a subcategory $\mathcal{J}_s$ for each $s\in S$. Moreover, 
$\mathcal{J}_E$ categorifies $J_E$ for each two-sided cell $E$, and
$\mathcal{J}_E$ is a \emph{multifusion category} in the sense of
\cite{EGNO} whenever $E$ is finite, which can happen for suitable cells even
when the ambient group $W$ is infinite. Similarly, $J_s$ is a \emph{fusion
category} whenever $J_s$ has finite rank. Multifusion and fusion
categories have rich connections with quantum groups (\cite{Kassel}),
conformal field theory (\cite{CFT}), quantum knot invariants (\cite{Turaev})
and topological quantum field theory (\cite{BKtcat}), so the categories
$\mathcal{J}_C$ and $\mathcal{J}_s (s\in S)$ are
interesting since they can potentially provide new examples of multifusion and
fusion categories.

Historically, the intimate connection between $J$ and $\mathcal{J}$ has been a major tool in the study of both
objects. For Weyl groups and affine Weyl groups, Lusztig (\cite{L4},
\cite{Lusztig_cells}) and
Bezrukanikov et al. (\cite{Bez1}, \cite{Bez2}, \cite{Bez3}) showed that there
is a bijection between the
two-sided cells in the group and unipotent conjugacy classes of
an algebraic group, and that the subcategories of  $\mathcal{J}$ corresponding
to the cells can be described geometrically, as categories of vector
bundles on a square of a finite set equivariant with respect to an algebraic
group. Using categorical results, they computed the structure constants in
$J$ explicitly. 
For other Coxeter systems, however, the nature of
$J$ or $\mathcal{J}$ seems largely unknown, partly because there is no known recourse
to advanced geometry. In this context, our paper may be viewed as an attempt
to understand the subalgebra $J_C$ of $J$ for arbitrary Coxeter systems from
a more combinatorial point of view. We hope to
understand the structure of $J_C$ by examining the multiplication rule in
$J_C$, then, in some cases, use our knowledge of $J$ to deduce the
structure of $\mathcal{J}$. This idea is further discussed in
Section \ref{sec:background}.

The main results of the paper fall into two sets. First, we
describe some connections between the {Coxeter diagram} $G$ of an arbitrary
Coxeter system $(W,S)$ and the subregular $J$-ring of $(W,S)$. The
first result in this spirit describes $J_C$ in terms of $G$ for all 
\emph{simply-laced} Coxeter systems. Recall that given any vertex $s$ in
$G$, the fundamental group $\Pi_s(G)$ of $G$ based at $s$ is the group
consisting of all homotopy equivalence classes of walks
on $G$ that start and end at $s$, equipped with concatenation as the group
operation. We generalize this notion to define the \emph{fundamental
groupoid} $\Pi(G)$ of $G$ as the set of homotopy equivalence classes of
all walks on $G$, equipped with concatenation as a partial binary
operation (see Section \ref{sec:simply-laced}). We define the groupoid algebra of $\Z\Pi(G)$ of $\Pi(G)$ by
mimicking the construction of a group algebra from a group, and we prove the following theorem. 

\begin{theorem}
  \label{simply-laced}
  Let $(W,S)$ be an any simply-laced Coxeter system, and let $G$ be its Coxeter
  diagram. Let $\Pi(G)$ be the fundamental groupoid of $G$, let
  $\Pi_s(G)$ be the fundamental group of $G$ based at $s$ for any $s\in S$, let
  $\Z\Pi(G)$ be the groupoid algebra of $\Pi(G)$, and let $\Z\Pi_s(G)$ be the
  group algebra of $\Pi_s(G)$. Then $J_C\cong \Z\Pi(G)$ as based rings, and
  $J_s\cong \Z\Pi_s(G)$ as based rings for all $s\in S$.  
\end{theorem}
\noindent The key idea behind the theorem is to find a correspondence between basis
elements of $J_C$ and classes of walks on $G$. The correspondence then
yields explicit formulas for the claimed isomorphisms.

In our second result, we study the case where $G$ is \emph{oddly-connected}.
Here by oddly-connected we mean that each pair of distinct vertices in
$G$ are connected by a path involving only edges of odd weights.
\begin{theorem}
  \label{oddly-connected}
  Let $(W,S)$ be an oddly-connected Coxeter system. Then 
  \begin{enumerate}
    \item $J_s\cong J_t$ as based rings for all $s,t\in S$. 
    \item $J_C\cong
      \mathrm{Mat}_{S\times S}(J_s)$ as based rings for all $s\in S$. In particular,
      $J_C$ is Morita equivalent to $J_s$ for all $s\in S$.
  \end{enumerate}
\end{theorem}
\noindent Once again, we will give explicit isomorphisms between the
algebras by using $G$. 

In a third result, we describe all \emph{fusion rings} that appear in the form
$J_s$ for some Coxeter system $(W,S)$ and some choice of $s\in S$. We show that any
such fusion ring is isomorphic to a ring $J_{s'}$ associated to a dihedral
system, which is in turn isomorphic to the \emph{odd part} of a
\emph{Verlinde algebra} associated to the Lie group $SU(2)$ (see Definition
\ref{verlinde def} and Corollary
\ref{fusion verlinde}).
\begin{theorem}
  \label{fusion J}
  Let $(W,S)$ be a Coxeter system such that $J_s$ is a fusion
  ring for some $s\in S$. Then there exists a dihedral Coxeter system
  $(W',S')$ such that $J_s\cong
  J_{s'}$ as based rings for either $s'\in S'$.
\end{theorem}

In our second set of results, we focus on certain specific Coxeter
systems $(W,S)$ whose Coxeter diagrams involve edges of weight $\infty$, and
show that for suitable choices of $s\in S$, $J_s$ is isomorphic to a \emph{free
fusion ring} in the sense of \cite{Banica}. A free fusion ring can be
described in terms of the data of its underlying \emph{fusion set}, and we
describe these data explicitly for each of our examples.
Furthermore, each free fusion ring we discuss is isomorphic to the
Grothendieck ring of the category of
representations of a known \emph{partition quantum group} $\mathbb{G}$, and we
will identify the group $\mathbb{G}$ in all cases.
Our main theorems appear as Theorem \ref{unitary} and
Theorem \ref{amalgamate} in sections \ref{sec:example2} and
\ref{sec:example3}, but we omit their technical statements for the
moment.

All the results mentioned above rely heavily on the following theorem, which
says that a combinatorial factorization of reduced words into \emph{dihedral segments} (see Definition \ref{dihedral
segments def}) carries over to a factorization of
basis elements in $J_C$.
\begin{theorem}[Dihedral factorization]
  \label{dihedral factorization}
  Let $x$ be the reduced word of an element in $C$, and let $x_1,
  x_2,\cdots, x_l$ be the dihedral segments of $x$. Then
  \[
    t_x=t_{x_1}\cdot t_{x_2} \cdot \cdots \cdot t_{x_l}.
  \]
\end{theorem}

The rest of the article is organized as follows. We quickly review the
construction of the asymptotic Hecke algebra $J$ from a Coxeter system in Section 2. In
Section 3, we define
the algebras $J_C$ and $J_{s} (s\in S)$ and describe their structure
as based rings. Kazhdan--Lusztig cells play an important role in this section.  We prove
Theorem \ref{dihedral factorization} and discuss the computation of
$J_C$ in Section 4. In Section 5, we prove our results on the
connections between $J_C$ and Coxeter diagrams. Finally, we
discuss our second set of results in Section 6, where we prove that certain
rings $J_s$ are free fusion rings.

%Let us mention some of our other endeavors in understanding $J$ and $J_C$.
%First, we are currently attempting to develop a theorem similar to Theorem
%\ref{dihedral factorization} for more general Kazhdan--Lusztig cells, namely for
%celle consisting of \emph{fully-commutative} (\cite{fc-finite}) elements of
%$W$. We have started some computations in this case and believe such a theorem
%can be proved. Second, using Theorem \ref{dihedral factorization} again, we
%have found an infinite family of Coxeter systems $(W,S)$ where for suitable
%$s\in S$, $J_s$ is isomorphic to a \emph{free fusion semiring}. Free fusion
%semirings appear in operator algebra theory as the Grothendieck rings of categories $\mathcal{C}$ of representations of so-called
%\emph{partition quantum groups}  (\cite{Banica},
%\cite{FW}, \cite{Freslon-1}, \cite{Freslon-2}), and we will prove the
%isomorphisms and study connections between $\mathcal{J}$ and
%$\mathcal{C}$ in an upcoming article \cite{Xu2}, where we try to deduce results
%on $\mathcal{J}$ by lifting results on $J$ to the categorical level. Finally, we
%we are seeking more interesting examples of rings that appear as $J_C$ or
%$J_s$ ($s\in S$). To do
%so, we have developed {\tt SageMath}
%(\cite{sagemath}) code for computation of in $J_C$, and we are doing
%experimental calculations of $J_C$ for Coxeter systems of small rank. 
%The code is available at \cite{mycode}.
%\tableofcontents
\section{Asymptotic Hecke algebras}
\label{sec:prelim}
We recall the construction of asymptotic Hecke algebras from
Coxeter systems in this section. Out main references are \cite{BB} and \cite{LG}. In
particular, when we define Hecke algebras we use a normalization with
base ring $\Z[v,v\inverse]$ and with quadratic relations $(T_s-v)(T_s+v\inverse)=0$
for all simple reflections $s\in S$. 

\subsection{Coxeter systems}
\label{sec:coxter systems}
A \emph{Coxeter system} is a pair $(W,S)$ where $S$ is a finite set equipped with a map
$m: S\times S\ra \Z_{\ge 1}\cup \{\infty\}$ and $W$ is the
group presented by
\begin{equation*}
  \label{eq:coxeter definition}
  W=\langle S\;\vert\;(st)^{m(s,t)}=1,\; \forall
  s,t\in S\rangle.
\end{equation*}
Here, $W$ is called a \emph{Coxeter group}, $S$ is called its set of
\emph{simple reflections}, and the map $m$ is required to satisfy that $m(s,s)=1$ and
$m_{s,t}=m_{t,s}\ge 2$ for all distinct elements $s,t\in S$. 
The data of
a Coxeter system $(W,S)$ can be encoded via a weighted, undirected graph $G$ called
the \emph{Coxeter diagram} of $W$. By definition, $G$ has $S$ as its vertex
set, and for any $s,t\in S$, the pair $\{s,t\}$ forms an edge in $G$ exactly when
$m(s,t)\ge 3$, in which case the edge has weight $m(s,t)$. When drawing
$G$, we label each edge with its weight except for those with weight 3. We say
$(W,S)$ is \emph{simply-laced} if all edges of $G$ are unlabeled.

Let $\langle S \rangle$ be the free monoid on $S$. Then elements of $W$ are
naturally represented by words in $\ip S$. Of the words representing an element
$w$, we call each word of minimal length a \emph{reduced word} of $w$. We
call that minimal length the \emph{length} of $w$ and denote it by $l(w)$.

For future use, we recall a few facts regarding reduced words below. First, note that since
$(ss)^{m(s,s)}=(ss)^1=s^2=1$ in $W$ for each
$s\in S$, the relation $(st)^{m(s,t)}=1$ is equivalent to the relation  
$sts\cdots
=tst\cdots$ where both products have $m(s,t)$ factors. Let us call each of
these
products an \emph{$(s,t)$-braid} and
call the action of replacing one $(s,t)$-braid with the other a \emph{braid
move}. 
Then we have the following
fundamental result on reduced words.

\begin{proposition}[Matsumoto-Tits Theorem. \cite{LG}, Theorem 1.9]
  \label{Matsumoto}
  Any two reduced words of an element in $W$ can be obtained from each
  other by a finite sequence of braid moves.  
\end{proposition}

Next, for any $w\in W$, we define the \emph{left descent set} and \emph{right descent set}
of $w$ to be the sets
\[
  \mathcal{L}(w)=\{s\in S: l(sw)<l(w)\}\quad\text{and}\quad
  \mathcal{R}(w)=\{s\in S: l(ws)<l(w)\},
\]
respectively. Descent sets can be characterized in terms of reduced words
as follows.

\begin{proposition}[\cite{BB}, Corollary 1.4.6]
  \label{descent}
  Let $s\in S$ and $x\in W$. Then
  \begin{enumerate}
  \item $s\in \mathcal{L}(w)$ if and only if $w$ has a reduced
  word beginning with $s$;
  \item $s\in \mathcal{R}(w)$ if and only if $w$ has a reduced
  word ending with $s$.
\end{enumerate}
\end{proposition}

Finally, we recall that each Coxeter group admits a partial order $\le$ called
the \emph{Bruhat order}. Define a \emph{subword} of any word $s_1s_2\cdots
s_k\in \ip S$ to be a word of the form $s_{i_1}s_{i_2}\cdots s_{i_l}$ where
$1\le i_1<i_2< \cdots < i_l\le k$. Then the Bruhat order can also be
characterized in terms of reduced word, in the following way.

\begin{proposition}[\cite{BB}, Corollary 2.2.3]
  \label{subword propositionerty}
  Let $x,y\in W$. Then the following are equivalent:
  \begin{enumerate}
  \item $x\le y$;
\item every reduced word for
  $y$ contains a subword that is a reduced word for $x$;
\item some reduced word for $y$ contains a subword that is a
  reduced word for $x$. 
\end{enumerate}
\end{proposition}

\subsection{Hecke algebras}
\label{sec:hecke algebras}

Let $(W,S)$ be an arbitrary Coxeter system, and let $\mathcal{A}=\Z[v,v\inverse]$. Following \cite{LG}, we define the
\emph{Hecke algebra} of $(W,S)$ to be the unital
$\mathcal{A}$-algebra $\calh$ generated by the set $\{T_s:s\in S\}$ subject to
the relations 
\begin{equation} \label{eq:quadratic}
  (T_s-v)(T_s+v\inverse)=0 
\end{equation}
for all $s\in S$ and the relations
\begin{equation}
  T_sT_tT_s\cdots=T_tT_sT_t\cdots  \label{eq:hecke-braid}
\end{equation}
for all $s,t\in S$, where both sides have $m(s,t)$ factors. 

Let $x\in W$, let $s_1s_2\cdots s_k$  be any reduced word of $x$, and set
$T_x:=T_{s_1}\cdots T_{s_k}$. Note that all reduced words of
$x$ produce the same element as $T_x$ by Proposition \ref{Matsumoto} and
Equation \eqref{eq:hecke-braid}, therefore $T_x$ is well-defined. Indeed, it is
well-known that
 the set $\{T_x:x\in W\}$ forms an $\cala$-basis, called the \emph{standard
 basis}, of
$\calh$. 

It is easy to check that there is a unique ring homomorphism $\bar{}: H\ra H$
such that $\bar v=v\inverse$ and $\bar T_{s}=T_s\inverse$, and that
$\bar{}$\, is in fact an involution which sends $T_w$ to
$T_{w\inverse}^{-1}$. Let $\cala_{<0}=\sum_{n:n<0}\Z
v^n,\cala_{\le 0}=\sum_{n:n\le 0}\Z v^n$, $H_{<0}=\sum_{w\in W}\cala_{<0} T_w$
and $H_{\le 0}=\sum_{w\in W}\cala_{\le 0}T_w$. Then the following holds. 

\begin{proposition}[\cite{LG}, Theorem 5.2]
  For $w\in W$, there exists a unique element $c_w\in H_{\le 0}$ such that 
  $\bar c_w=c_w$ and $c_w=T_w\mod H_{<0}$. Moreover, the set $\{c_w:w\in W\}$
  forms an $\cala$-basis of $H$.
\end{proposition}
\noindent We call the basis $\{c_x:x\in W\}$ the
\emph{Kazhdan--Lusztig basis} of $H$, and define the \emph{Kazhdan--Lusztig polynomials} to
be the elements $p_{x,y}\in \cala_{\le 0}$ for which
\[
  c_y=\sum_{x\in W} p_{x,y} T_x
\]
for all $x,y\in W$. It is worth noting that these definitions differ slightly
from the definitions of the Kazhdan--Lusztig basis $\{C_w:w\in W\}$ and the
Kazhdan--Lusztig polynomials $\{P_{x,y}:x,y\in W\}$ in the paper
\cite{KL} where these
notions were first introduced.  However, the difference is not essential and
only a result of the
difference in the normalizations of $H$, and it is easy to translate between the
two conventions. In particular, the base ring of
$H$ is the ring $\Z[q]$ and $P_{x,y}\in \Z[q]$ for any $x,y\in W$ in \cite{KL},
but we may obtain
$p_{x,y}$ from $P_{x,y}$ by substituting $q$ by $v^2$ in
$P(x,y)$ and then multiplying the result by $v^{l(x)-l(y)}$.

\begin{notation}
  From now on we will mention the phrase ``Kazhdan--Lusztig'' numerous times. We
  will often abbreviate it to ``KL''.
\end{notation}

The KL basis and the KL polynomials enjoy a remarkable ``positivity'' property.
More precisely, let $h_{x,y,z}\in \Z[v,v\inverse]$ ($x,y,z\in W$) be the
elements such that  
 \begin{equation}
  c_xc_y=\sum_{z\in W}h_{x,y,z}c_z,
  \label{eq:h polynomials}
\end{equation}
then by positivity we mean that the elements $h_{x,y,z}$ and $p_{x,y}$ always
have non-negative integer coefficients. This result, known as the
\emph{Kazhdan--Lusztig positivity conjecture}, was first conjectured in \cite{KL} and only recently proven in
\cite{EW}:

\begin{proposition}[\cite{EW}, Corollary 1.2]
  \label{positivity}
  {\hspace{2em}}
  \begin{enumerate}
    \item $p_{x,y}\in \N[v\inverse]$ for all $x,y\in W$.
    \item $h_{x,y,z}\in \N[v,v\inverse]$ for all $x,y,z\in W$.
  \end{enumerate}
\end{proposition}

Let us record a multiplication formula of the KL basis for
future use. For $x,y\in W$, let $\mu_{x,y}$ denote the coefficient of $v^{-1}$
in $p_{x,y}$. Then the following holds.
\begin{proposition}[\cite{LG}, Theorem 6.6, Corollary 6.7]
  \label{KL basis mult}
  Let $y\in W$, $s\in S$, and let $\le$ be the Bruhat order on $W$. Then
  \begin{eqnarray*}
    c_s c_y&=& 
    \begin{cases}
      (v+v\inverse) c_y&\quad\qquad\text{if}\quad sy<y\\
      c_{sy}+\sum\limits_{z:sx<x<y} \mu_{x,y}c_z&\quad\qquad \text{if} \quad sy>y
    \end{cases},\\
    &&\\
    c_y c_s&=& 
    \begin{cases}
      (v+v\inverse) c_y&\quad \text{if}\quad ys<y\\
      c_{ys}+\sum\limits_{x:xs<x<y} \mu_{x^{-1},y^{-1}}c_x&\quad \text{if} \quad sy>y
    \end{cases}.
    \\
  \end{eqnarray*}
\end{proposition}

\subsection{Asymptotic Hecke algebras}
\label{aha}
Let $(W,S)$ be a Coxeter system with
Hecke algebra $H$, let $h_{x,y,z}$ be as in 
Equation \eqref{eq:h polynomials}, and let $f_{x,y,z}$ be the structure
constants of $H$ with respect to the standard basis of $H$, so that
\[
  T_xT_y=\sum_{z\in W} f_{x,y,z}T_z
\]
for all $x,y\in W$. We say that $W$ is \emph{bounded} if there is a
nonnegative integer $N$ such that $v^{-N}f_{x,y,z}\in \Z[v\inverse]$ for all
$x,y,z\in W$.  All finite Coxeter groups and affine Weyl groups are known to be
bounded, and it is conjectured by Lusztig that all Coxeter groups are bounded
(see \cite{LG}, Conjecture 13.4). The conjecture is crucial for the study of
asymptotic Hecke algebras, and we shall assume that it is true
for the rest of the paper:
\begin{assumption}
  \label{bdd}
  From now on, we assume that all Coxeter groups are bounded.
\end{assumption}

Let $z\in W$. Under Assumption \ref{bdd}, it is known (see \cite{LG}, Chapter 13) that there exists a unique integer
$\mathbf{a}(z)\ge 0$ that satisfies the conditions 
\begin{enumerate}
  \item[(a)] $h_{x,y,z}\in v^{\mathbf{a}(z)}\Z[v\inverse]$ for all $x,y\in W$,

  \item[(b)] $h_{x,y,z}\not\in v^{\mathbf{a}(z)-1}\Z[v\inverse]$ for some $x,y\in W$. 
\end{enumerate}
For all $x,y\in W$, we define $\gamma_{x,y,z\inverse}$ to be the integer such that
\[
  h_{x,y,z}=\gamma_{x,y,z\inverse}v^{\la(z)}\,\mod v^{\la(z)-1}\Z[v\inverse].
\]

The \emph{asymptotic Hecke algebra} of $(W,S)$ is defined to be the free
abelian group $J=\oplus_{w\in W}\Z t_w$, with multiplication declared by
\begin{equation}
  \label{eq:J mult}
  t_{x}t_y=\sum_{z\in W} \gamma_{x,y,z\inverse} t_z
\end{equation}
for all $x,y\in W$. The multiplication is well-defined (i.e.,
$\gamma_{x,y,z\inverse}=0$ for all but finitely many $z\in W$ for all $x,y\in
W$) and associative, hence $J$ is indeed a ring (see \cite{LG}, Section 18.3). Henceforth,
we will also often simply call $J$ the \emph{$J$-ring} of $(W,S)$.

Recall that $p_{x,y}\in \cala_{\le 0}= \Z[v\inverse]$ for any $x,y\in
W$. Thus, for each $y\in W$, there exists a unique non-negative integer,
$\Delta(y)$, such that
\begin{equation}
  \label{eq:Delta}
  p_{1,y}\in n_yv^{-\Delta(y)}+v^{-\Delta(y)-1}\Z[v\inverse]
\end{equation}
for some $n_y\neq 0$.  Let
\begin{equation}
  \label{eq:cald}
  \cald=\{y\in W:\la(y)=\Delta(y)\}. 
\end{equation}
Then it is known that $d^2=1$ for all $d\in \cald$, and we call $\cald$ the set of
\emph{distinguished involutions} of $W$. 
As mentioned in Remark \ref{unit remark}, the set $\cald$ is intimately
related to the multiplicative structure of $J$. We will recall more facts about $\cald$
in Section \ref{sec:cells}.

\section{The subregular $J$-ring}
In this section, we recall the definition and some properties of
Kazhdan--Lusztig cells, then define our main object of study---the subregular
$J$-ring. We will also describe
the structure of the subregular $J$-ring as a based ring. 

The notations of Section \ref{sec:prelim} and Assumption \ref{bdd} remain in
force.
\subsection{Kazhdan--Lusztig cells}
\label{sec:cells}
For each $x\in W$ let $D_x:\calh\ra\cala$ be the linear map such that 
\[
  D_x(c_{y})=\delta_{x,y}
\]
for all $y\in W$. For $x,y\in W$, 
\begin{enumerate}
  \item define $x\prec_{L} y$ if $D_x(c_sc_{y})\neq 0$ for some $s\in S$;
  \item define $x\le_L y$ if there is a sequence $x=z_1,z_2,\cdots, z_n=y$ in
    $W$ such that $z_i\prec_L z_{i+1}$ for all $1\le i\le n-1$;
  \item define $x\sim_L y$ if $x\le_L y$ and $y\le_L x$.
\end{enumerate}
Then $\sim_L$ is an equivalence relation. We call the equivalence classes the
\emph{left Kazhdan--Lusztig cells} of $W$. We may similarly define
\emph{right Kazhdan--Lusztig cells} and \emph{two-sided Kazhdan--Lusztig
cells}, where for the latter we replace the ``$\le_L$'' in Step (2) by declaring
$x\le_{LR} y$ if
there exists a sequence $x=z_1,\cdots, z_n=y$ in $W$ such that
$z_{i}\prec_L z_{i+1}$ or $z_{i}\prec_R z_{i+1}$ for all $1\le i\le n-1$.
Clearly, each 2-sided KL cell is a union of left cells as
well as a union of right cells. 

KL cells enjoy many nice properties that are important to this paper. 
Let us first observe the following.
\begin{proposition}[\cite{LG}, Lemma 8.2]
  \label{KL order and ideals}
  Let $y\in W$. Then
  \begin{enumerate}
\item the set $\calh_{\le_L y}:=\oplus_{x:x\le_L y} \cala c_x$ is a
  left ideal of $\calh$;
\item the set $\calh_{\le_R y}:=\oplus_{x:x\le_R y} \cala c_x$ is a right ideal
  of $\calh$;
\item the set $\calh_{\le_{LR} y}:=\oplus_{x:x\le_{LR} y} \cala c_x$ is a
  two-sided ideal of $\calh$.
\end{enumerate}
\end{proposition}

\begin{proof}
  The definition of $\prec_L$ guarantees that $H_{\le y}$ is closed under left multiplication
  by $c_s$ for each $s\in S$. Since the elements $c_s$ generate $H$ by
  Proposition \ref{KL basis mult}, it follows that $H_{\le y}$ is a left
  ideal, proving (1). The proofs of (2) and (3) are similar. 
\qed\end{proof}

Next, we recall two  compatibility  results on cells and  the inverse map on $W$. 
\begin{proposition}[\cite{LG}, Section 8.1]
  \label{inverse and cells}
  The map $w\mapsto w\inverse$
  takes left cells in $W$ to right cells, right cells to left cells, and
  2-sided cells to 2-sided cells. 
\end{proposition}

\begin{proposition}[\cite{LG}, Conjecture 14.2]
For any $w\in W$, we have $w\sim_{LR} w\inverse$.
\end{proposition}

\begin{remark}
 The book \cite{LG} studies Hecke
  algebras in a more general setting than ours, namely with possibly \emph{unequal
  parameters}. The above statement appears as a
  conjecture in Section 14 of \cite{LG} but is  known to be true in
  our setting, which is called the \emph{equal parameter} or the \emph{split}
  case in the book. The same is true for all other statements we shall quote from
  Section 14 of \cite{LG}. The proofs of the statements rely on Proposition
  \ref{positivity} and can be found in Chapter 15 of \cite{LG}.  
\end{remark} 

KL cells have close connections with distinguished involutions and the
structure constants of the $J$-ring. In particular, we have the following
facts.
\begin{proposition}[\cite{LG}, Conjecture 14.2]\hfill
  \begin{enumerate}
  \item 
  Each left KL cell $\Gamma$ of $W$ contains a unique $d\in \cald$.
  We have $\gamma_{x\inverse,x,d}\neq 0$ for all $x\in\Gamma$.
\item  If $\gamma_{x,y,z}\neq 0$ for $x,y,z\in W$, then $x\sim_L
  y\inverse, y\sim_L z\inverse, z\sim_L x\inverse$. 
\item If $\gamma_{x,y,d}\neq 0$ for $x,y\in W$ and $d\in \cald$, then
  $x=y\inverse$. 
\item For each $x\in W$, there is a unique $d\in \cald$ such that
  $\gamma_{x,x\inverse,d}\neq 0$.
  \end{enumerate}
  \label{gamma and cells}
\end{proposition}

Finally, we explain how cells give rise to subalgebras of the $J$-ring. Recall that
$J=\sum_{t\in W}\Z t_w$ as a group. For any subset $X\se W$, define
$J_X:=\sum_{w\in X}\Z t_w$, the subgroup supported on $X$.  
Then the following
holds.

\begin{proposition}[\cite{LG}, Section 18.3]\hfill
  \label{subalgebra}
  \begin{enumerate}
    \item Let $\Gamma$ be any left KL cell in $W$, and let $d$ be the
      unique element of $\Gamma\cap\cald$. Then the subgroup $J_{\Gamma\cap \Gamma\inverse}$ is a unital subalgebra of $J$; its unit is $t_d$.  
    \item Let $E$ be any two-sided cell
      $E$ in $W$. Then $J_E$ is a subalgebra of $J$. Further, if
     $E \cap \cald$ is finite, then
      $J_E$ is a unital algebra with unit $\sum_{d\in E\cap \cald}
      t_d$.
\item We have a
      direct sum decomposition $J=\oplus_{E\in \mathcal{C}}J_E$ of algebras, where
      $\mathcal{C}$ is the collection of all two-sided KL cells of $W$.  
    \end{enumerate} 
  \end{proposition}

\subsection{The subregular $J$-ring}
Consider the following proposition.
\begin{proposition}[\cite{subregular}, Theorem 3.8]
  \label{subregular}
  Let $C$ denote the set of all non-identity elements in $W$ with a unique reduced
  word. For each $s\in S$, let $\Gamma_s$ be the sets of
  elements in $C$ whose reduced
  word ends in $s$. Then $C$ is a two-sided Kazhdan--Lusztig cell of $W$, and $\Gamma_s$
  is a left Kazhdan--Lusztig cell of $W$ for each $s\in S$.
\end{proposition}

From now on, we shall call the above set $C$ the \emph{subregular
cell} of $W$ and reserve the notation $C$ for this cell.
We call the subalgebra $J_C$ of $J$ the \emph{subregular
  $J$-ring} of $(W,S)$. For each $s\in S$, we write
  $J_s:=J_{\Gamma_s\cap\Gamma_s\inverse}$.  The rest of the paper is dedicated
  to the study of the algebras $J_c$ and $J_s (s\in S)$.

Recall the functions $\la, \Delta: W\ra \Z_{\ge 0}$ from Section
\ref{aha}. To us, an important feature of the subregular cell is that it
is exactly the set of elements of $\la$-value 1, thanks to the following
facts.

\begin{proposition}[\cite{LG}, 13.7, 14.2] \label{a and cells}
  Let $x,y\in W$. Then 
  \begin{enumerate} 
    \item   $\la(x)\le \Delta(x)$.  
    \item $\mathbf{a}(x)\ge 0$, where
      $\la(x)=0$ if and only if $x$ equals the identity element of $W$.  
    \item If $x\le_{LR} y$, then $\la(x)\ge \la(y)$.
      Hence, if $x\sim_{LR} y$, then $\la(x)=\la(y)$.  
    \item If $x\le_{LR} y$ and
      $\la(x)=\la(y)$, then $x\sim_{LR} y$.  
  \end{enumerate} 
\end{proposition}

\begin{corollary} 
  \label{a=1} 
  Let $x\in W$. Then $\la(x)=1$ if and only if  $x\in C$.
\end{corollary}
\begin{proof} 
  We first prove that $\la(x)=1$ for all $x\in C$. Let $s\in S$. Using Equation \ref{eq:quadratic}, it is easy to check
  that $c_s=T_s+v\inverse$, therefore $\Delta(s)=1$  and $\la(s)\le 1$ by part
  (1) of Proposition  \ref{a and cells}. Meanwhile, Part (2) implies $\la(s)\ge 1$, therefore 
  $\la(s)=1$. Since $s$ is clearly in $C$, Part (3) implies that $\la(x)=1$ for all
  $x\in C$. 

  It remains to prove that $\la(x)\neq 1$ for any $x\in W\setminus C$. Let
  $x\in W\setminus C$. Then either $x$ is the group identity and
  $\la(x)=0$ by Part (1) of Proposition \ref{a and cells}, or $x$ has a reduced expression $x=s_1s_2\cdots s_k$ with $k>1$
  and each $s_i\in S$.  In the latter case, $x\le_L s_k$ by Proposition \ref{KL basis
  mult}, so $\la(x)\ge \la(s_k)=1$. Meanwhile, since  $x\not\sim_{LR} s_k$,
  Part (4) of Proposition  \ref{a and cells} implies that $\la(x)\neq \la(s_k)$, therefore $\la(x)>1$ and we are done.  \qed\end{proof}

\begin{remark}
  \label{C distinguished}
  By the proof of the corollary, the distinguished involutions in
  $C$ are exactly the simple reflections of the Coxeter system.
  Consequently, $J_C$ has unit
  $\sum_{s\in S}t_s$ and $J_s$ has unit $t_s$ for each $s\in S$.  
\end{remark}

  \subsection{Based ring structure of $J_C$}
  \label{sec:based rings}
Thanks to the positivity of its structure constants and certain other
properties, the subregular $J$-ring is a 
based ring in the sense of \cite{EGNO}. We prove this claim now. 

Let us first recall the relevant
definitions from
Chapter 3 of \cite{EGNO}. 
\begin{definition}[$\Z_{+}$-ring]
    Let $A$ be a ring which is free as a $\Z$-module.
    \begin{enumerate}
      \item A \emph{$\Z_+$-basis} of $A$ is a basis $B=\{t_i\}_{i\in I}$ such
        that for all $i,j\in I$, 
        $t_it_j=\sum_{k\in I}c_{ij}^k t_k$ where $c_{ij}^k\in \Z_{\ge 0}$ for all
        $k\in I$.
      \item A \emph{$\Z_+$-ring} is a ring with a fixed $\Z_+$-basis and with
        identity 1 which is a nonnegative linear combination of the basis
        elements.
      \item A \emph{unital} $\Z_+$-ring is a $\Z_+$ ring such that 1 is a basis
        element.
    \end{enumerate}
  \end{definition}

  Let $A$ be a $\Z_+$-ring, and let $I_0$ be the set of $i\in I$ such that
  $t_i$ occurs in the decomposition of 1. We call the
  elements of  $I_0$ the \emph{distinguished index set}. Let $\tau: A\ra \Z$ denote the
  group homomorphism defined by 
  \[
    \tau(t_i)=
    \begin{cases}
      1 &\quad \text{if}\quad i\in I_0,\\
      0 &\quad \text{if}\quad i\not\in I_0.
    \end{cases}
  \]

  \begin{definition}[Based rings]
    \label{based rings def}
    A $\Z_+$-ring $A$ with a basis $\{t_i\}_{i\in I}$ is called a
    \emph{based ring} if there exists an involution $i\mapsto i^*$ such that
    the induced map 
    \[
      a=\sum_{i\in I} c_it_i\mapsto a^*:=\sum_{i\in I} c_it_{i^*}, c_i\in \Z
    \]
    is an anti-involution of the ring $A$, and 
    \begin{equation}\label{eq:based ring}
      \tau(t_it_j)=
      \begin{cases}
        1 &\quad \text{if}\quad i=j^*,\\
        0 &\quad \text{if}\quad i\neq j^*.
      \end{cases}
    \end{equation}
    We denote the data of the based ring by $(A,I,I_0,*)$.
  \end{definition}

  \begin{definition}[Multifusion rings and fusion rings]
    \label{fusion def}
    A \emph{multifusion ring} is a based ring of finite rank. A \emph{fusion
    ring} is
    a unital based ring of finite rank. 
  \end{definition}

We shall reserve the meaning of $I,I_0$ and $*$ from the previous definitions
  throughout the paper.
   We now describe the based ring structure of the subregular
  $J$-ring.  
  
  \begin{proposition}
    \label{based structure}\hfill
    \begin{enumerate}
      \item  Let $E$ be any 2-sided KL cell in $W$ that contains finitely many
        distinguished involutions. Then the algebra $J_E$ is a based ring
      with
   $I=E$, $I_0=E\cap \cald$ and $x^*=x\inverse$ for all $x\in I$.
      \item Let $\Gamma$ be any left KL cell in $W$, and let $d$ be the unique
        element in $\Gamma\cap \cald$. Then $J_{\Gamma\cap\Gamma\inverse}$ is a
        unital based ring with $I=\Gamma\cap\Gamma\inverse$, $I_0=\{d\}$ and
        $x^*=x\inverse$ for all $x\in I$.
    \end{enumerate}
  \end{proposition}
  \begin{proof}
    (1) The set $\{t_x\}_{x\in E}$ forms a $\Z_+$-basis of $J_E$ by the
    definition of $J_E$, and $J_E$ is $\Z_+$-ring with distinguished index
    set $E\cap \cald$ by Part (3) of Corollary \ref{subalgebra}. The fact that $x\mapsto x\inverse$ induces an anti-involution holds
    because $\gamma_{x,y,z}=\gamma_{y\inverse,x\inverse,z\inverse}$ by symmetry (see
    Proposition 13.9 of \cite{LG}). Finally, Equation \eqref{eq:based ring}
    follows from Proposition \ref{gamma and cells}. We have now proven the
    claim.

    (2) The proof is similar to the previous part, with the only
    difference being that $J_{\Gamma\cap\Gamma\inverse}$ is unital with
    $I_0=\{d\}$, for $t_d$ is its unit by Part (1) of Corollary \ref{subalgebra}.
  \qed\end{proof}

  \begin{corollary}
    \label{subregular base}
    Let $(W,S)$ be a Coxeter system (recall that $S$ is finite by definition). Let $C, \Gamma_s, J_C$ and $J_s$ be
    as before. Then
    \begin{enumerate}
      \item $J_C$ is a based ring with $I=C$, $I_0=S$ and $x^*=x\inverse$ for
        all $x\in I$.
      \item For each $s\in S$, $J_s$ is a based ring with
        $I=\Gamma_s\cap\Gamma_s^{-1}$, $I_0=\{s\}$
        and         $x^*=x\inverse$ for all $x\in I$.
    \end{enumerate}
  \end{corollary}
  \begin{proof}
    This is immediate from Proposition \ref{based structure} and Remark
    \ref{C distinguished}.  \qed\end{proof}

  Let us formulate the notion of an isomorphism of based rings.
  Naturally, we define it to be a ring isomorphism that respects all the
  additional structures of a based ring.

  \begin{definition}[Isomorphism of Based Rings]
    Let $(A,I,I_0,*)$ and $(B,J,J_0,*)$ be the data of two based rings.  We define an
    \emph{isomorphism of based rings} from $A$ to $B$ to be a unit-preserving
    ring isomorphism $\Phi: A\mapsto B$ such that $\Phi(t_i)=t_{\phi(i)}$ for
    all $i\in I$, where $\phi$ is a bijection from $I$ to $J$ such that
    $\phi(I_0)=J_0$ and $\Phi(t_i^*)=(\Phi(t_i))^*$ for all $i\in I$.
  \end{definition}

All the main results of the paper will assert that $J_C$ or some $J_s$ is isomorphic to a
  certain ring as a based ring.

  \section{Computation of $J_C$}
  We develop an approach to compute the subregular $J$-ring in this section.
  To do so, we 
provide an algorithm to compute products of the basis elements of $J_C$ and
define a graph to help enumerate the elements of $C$. We will work out an
example at the end of the section.

\subsection{A filtration of $H$}
We hope to understand products of the form $t_x\cdot t_y$ where $x,y\in C$. 
By the construction of the $J$-ring, in order to do so we need to
carefully examine the product $c_x\cdot c_y$ in the Hecke algebra $H$.   The goal of this subsection
is to show that in fact it suffices to examine $c_x\cdot c_y$ in a subquotient of $H$ instead. This will prove to be a useful
simplification.

To define the said subquotient, view $H$ as a regular left module. By Proposition \ref{KL order and ideals} and  Proposition
\ref{a and cells}, $H$ admits a filtration of left submodules
\begin{equation*}
  \cdots\subset \calh_{\ge 2}\subset \calh_{\ge 1}\subset \calh_{\ge 0}=\calh
\end{equation*}
where
\[
  \calh_{\ge a}=\oplus_{w: \la(w)\ge a} \cala c_w
\]
for each $a\in \N$. It induces the quotient modules 
\begin{equation}
  \label{eq:quotient}
  \calh_a:=\calh_{\ge a}/\calh_{\ge a+1},
\end{equation}
where $\calh_a$ is spanned by the images of the elements $\{c_w: \la(w)=a\}$. In
particular, 
$\calh_1$ is spanned by the images of $\{c_w:w\in C\}$ by Proposition
\ref{a=1}. Thus, to compute a product $t_x\cdot t_y$ in $J_c$, it suffices to consider
the product $c_x\cdot c_y$ in $\calh_1$. More precisely, we 
have arrived at the following simplification.

\begin{corollary} \label{J_C shortcut} Let $x,y\in C$. Suppose \[
  c_xc_y=\sum_{z\in W} h_{x,y,z} c_z \] for $h_{x,y,z}\in \cala$. Then \[
  t_xt_y=\sum_{z\in T}\gamma_{x,y,z\inverse} t_z \] in $J_C$, where $T=\{z\in
    C: h_{x,y,z}\in n_z v+\Z[v\inverse] \;\text{for some}\; n_z\neq 0\}$.
  \end{corollary} \noindent

  \subsection{The dihedral case} 
  \label{sec:dihedral case} 
  Let $(W,S)$ be the dihedral Coxeter system
  with $S=\{s,t\}$ and $M:=m(s,t)\in
  \Z_{\ge 3}\cup \{\infty\}$. With Corollary \ref{J_C shortcut} in place, we
  are ready to compute the subregular $J$-ring of $W$. 
  
  Let us set up some notation. For $k\in \Z_{\ge 0}$ and $u,v\in S$, we will
  use $u_k v$ to denote the word $w$ of
  length $k$ that starts with $u$, alternates in $s$ and $t$, and ends
  in $v$. Of course, $k$ and $u$ automatically determine the value of $v$, and $k,v$
  determine the value of $u$, but we will keep the notation in cases where we wish to emphasize both
  the first and last letter of $w$. When such emphasis is not needed, we
  will omit one of $u$ and $v$. For example, we will write $s_3$ or ${}_3 s$ for $sts$ and
  $s_4$ or ${}_4t$ for $stst$.
  
  It is well-known that the sets $\Gamma_s=\{ {}_k s:1\le k<M\}$ and
  $\Gamma_t=\{ {}_k t:1\le k<M\}$ are both left KL cells of $W$ and their union
  forms the subregular cell (see Section 7 of \cite{LG}).  Thus, for
  $x,y\in C$, $x\sim_L y\inverse$ if and only if the reduced word of $x$ ends
  with the letter that the reduced word of $y$ starts with. To describe
  products of the form $t_x\cdot t_y$ in $J_C$, we may focus only on this case, because otherwise $t_x\cdot t_y=0$ by Equation \eqref{eq:J mult} and
  Part (2) of Proposition \ref{gamma and cells}. Thus, we need to consider the
  following products in the Hecke algebra.

  \begin{proposition}
    \label{dihedral KL mult}
    Suppose $x=u_k s$ and $y=s_l v$ for some $u,v\in \{s,t\}$ and
    $0<k,l<M$. For $d\in \Z$, let $\phi(d)=k+l-1-2d$.
    Then
    \[
      c_xc_y=c_{u_k s}c_{s_l {v}}=\varepsilon+
      (v+v\inverse)\sum\limits_{d=\max(k+l-M,0)}^{\min(k,l)-1}c_{u_{\phi(d)}
    v}  \]
  in $\calh$, where $\varepsilon= f\cdot c_{s_M}$ for some $f\in \cala$ if
  $M<\infty$ and $\varepsilon=0$
  otherwise.
\end{proposition}
\begin{proof}
  It is known that $c_{uv}=c_uc_v$ and $c_{u_{k+1}}=c_{u}c_{v_k}-c_{u_{k-1}}$ for any
  distinct $u,v\in S$ and $1<k<M$ (see Section 7 of \cite{LG}). The Proposition is then straightforward to
  prove by induction on $k$. We omit the details. 
\qed\end{proof}

The following is now immediate by Corollary \ref{J_C shortcut}.
\begin{corollary}
  \label{d3}
  Suppose $x={u}_k s$ and $y=s_l v$ for some $u,v\in \{s,t\}$ and $0< k,l <M$. 
  For $d\in \Z$, let $\phi(d)=k+l-1-2d$.
  Then in $J_C$, we have
  \begin{equation}
    \label{eq:summands}
    t_xt_y=t_{u_k s}t_{s_l v}=\sum\limits_{d=\max(k+l-M,0)}^{\min(k,l)-1}
    t_{u_{\phi(d)}v}.
  \end{equation}
\end{corollary}

The elements indexing the summands in Equation \eqref{eq:summands} are
determined by the following properties: their reduced words all start with the
same letter as $x$ and end
with the same letter as $y$, and their lengths can be obtained by the
following rule: consider the list $\abs{k-l}+1,
\abs{k-l}+3, \cdots, k+l-1$ of numbers of the same parity, then delete from it all numbers
$r$ with $r\ge M$ as well as their mirror images with respect to the point
$M$, i.e., delete $2M-r$.

The rule we just described is in fact well-known; it is the \emph{truncated
Clebsch--Gordan rule}. It governs the
multiplication of the basis elements of the \emph{Verlinde algebras of the
Lie group $SU(2)$}, which
appear as the Grothendieck rings of certain {fusion categories}
(see \cite{EK} and Section 4.10 of \cite{EGNO}). 
Since it will cause no confusion, we will refer to these algebras simply as
\emph{Verlinde algebras}.

\begin{definition}[Verlinde algebras, \cite{EK}]
  \label{verlinde def}
  Let $n\in \Z_{\ge 2}\cup\{\infty\}$. The \emph{$n$-th Verlinde algebra} is the
  free abelian group $\mathrm{Ver}_n= \oplus_{1\le k\le n-1} \Z
  L_k$, with multiplication defined by 
  \begin{equation}
    \label{Verlinde mult}
    L_k L_l=\sum\limits_{d=\max(k+l-n,0)}^{\min(k,l)-1}
    L_{\phi(d)}
  \end{equation}
  where $\phi(d)=k+l-1-2d$.
  We call the $\Z$-span of the elements $L_k$ where $k$ is an odd integer the
  \emph{odd part} of $\mathrm{Ver}_n$, and denote it by $\mathrm{Ver}_n^{\text
  odd}$. 
\end{definition}

Note that by the Equation \eqref{Verlinde mult}, $\mathrm{Ver}_n^{\text odd}$
is a subalgebra of
$\mathrm{Ver}_n$. Further, both $\mathrm{Ver}_n$ and $\mathrm{Ver}_n^{\text
odd}$ are unital based rings with unit $L_1$ and anti-involution $L_i\mapsto L_i$. 

  \begin{proposition}
    \label{dihedral verlinde}
    We have $J_s\cong \mathrm{Ver}_M^{\text
    odd}$ as based rings.
  \end{proposition}
  \begin{proof} 
  The set $\Gamma_s\cap\Gamma_s\inverse\se C$ consists exactly
  of the elements $s_k$ where $1\le k\le M-1$ and $k$ is odd, therefore the map
  $t_{s_k}\mapsto L_k$ is an isomorphism of based rings from $J_s$ to
  $\mathrm{Ver}_M^{\text odd}$ by equations
  \eqref{eq:summands} and \eqref{Verlinde mult}.
\qed\end{proof}

\begin{example}
  When $M=5$, we have 
  \[
    J_s\cong \mathrm{Ver}_5^{odd}\cong \Z t_{s}\oplus \Z t_{sts}
  \]
  where $t_s$ is the unit element and $t_{sts}^2=t_s$, hence both $J_s$ and
  $\mathrm{Ver}_5^{\text odd}$ are isomorphic to the \emph{Ising fusion ring} that
  arises from the Ising model of statistical mechanics. When $M=6$, 
  both $J_s$ and $\mathrm{Ver}_6^{\text odd}$ are isomorphic to
  the Grothendieck ring of the category of finite dimensional representations
  of the symmetric group $S_3$.
\end{example}

  \subsection{Dihedral Factorization}
We now prove Theorem \ref{dihedral factorization}, which is restated below. The theorem will provide a bridge between
 the products $t_x\cdot t_y\in J_C$ in a dihedral group and such products in a
  general Coxeter group. 

\begin{thm1.4}[dihedral factorization]
  Let $x$ be the reduced word of an element in $C$, and let $x_1,
  x_2,\cdots, x_l$ be the dihedral segments of $x$. Then
  \[
    t_x=t_{x_1}\cdot t_{x_2} \cdot \cdots \cdot t_{x_l}.
  \]
\end{thm1.4}

  We need to first define the term ``dihedral segments''. Since we will only be concerned with reduced words of elements  and no
  element $s\in S$ can appear consecutively in a reduced word, we make the following assumption.

\begin{assumption}
  \label{no rep}
  From now on, whenever we speak of a word in a Coxeter system,
  we assume that no simple reflection appears consecutively in the word.
\end{assumption}

\begin{definition}[Dihedral segments]
  \label{dihedral segments def}
 \noindent For any $x\in \ip{S}$, we define
  the \emph{dihedral segments} of $x$ to be the maximal contiguous subwords of
  $x$ involving two letters. 
\end{definition}
\noindent For example, suppose $S=\{1,2,3\}$ and $x=121313123$, then
$x$ has dihedral segments $x_1=121, x_2=13131, x_3=12, x_4=23$. 

We may think of
breaking a word into its dihedral segments as a ``factorization'' process. The
process can be easily reversed by taking a proper ``product'':
\begin{definition}[Glued product] 
  \label{glued product def}
  For any two words
  $x_1,x_2\in \ip{S}$
  such that $x_1$ ends with the same letter that $x_2$ starts with, say
  $x_1=\cdots st$ and $x_2=tu\cdots$, we define their \emph{glued product} to be the
  word $x_1 * x_2:=\cdots stu\cdots$ obtained by concatenating $x_1$ and
  $x_2$ then deleting one occurrence of the common letter. 
\end{definition}
\noindent The
operation $*$ is obviously associative. Furthermore, if $x_1,x_2,\cdots, x_k$
are the dihedral segments of $x$, then 
  $x=x_1*x_2*\cdots * x_k$. Theorem \ref{dihedral factorization} can be
  viewed as an algebraic counterpart of this combinatorial factorization. 
  
  We now prove  Theorem \ref{dihedral factorization}. We need the
  following well-known fact. 
  \begin{proposition}[\cite{KL}, 2.3.e]
    \label{extremal}
   Let $x,y\in W, s\in S$ be such that $x<y, sy<y, sx>x$.  
   Then $\mu(x,y)\neq 0$ if and only if $x=sy$; moreover, in this case, $\mu(x,y)=1$.
 \end{proposition}
  \begin{lemma}
    \label{truncation}
    Let $x=s_1s_2s_3\cdots s_k$ be the reduced word of an element in $C$. Let
    $x'=s_2s_3\cdots s_k$ and $x''=s_3\cdots s_k$ be the sequences obtained by
    removing the first letter and first two letters from $x$, respectively. Then in $\calh_1$,
    we have
    \[
      c_{s_1}c_{x'}=
      \begin{cases}
        c_{x''}  &\quad\text{if}\; s_1\neq s_3;\\
        c_{x}+c_{x''} & \quad\text{if}\; s_1=s_3.
      \end{cases}
    \]
  \end{lemma}

  \begin{proof}
    By Proposition \ref{KL basis mult} and Corollary \ref{a=1}, in $\calh_1$ we have
    \begin{equation*}
      c_{s}c_{x'}=c_x+\sum_{P}\mu_{z,x'}c_z
    \end{equation*}
    where $P=\{z\in C:s_1z<z<x'\}$. Let $z\in P$. Then by Proposition \ref{descent},
    the unique reduced word of $z'$ starts with $s_1$ and hence $s_2z<z$. Since
    $z<x',s_2x'<x'$ and $s_2z>z$, Proposition \ref{extremal} implies
    that $\mu_{z,x'}\neq 0$ if and only if $z=s_2x'=x''$ and that
    $\mu(z,x')=1$ in this case. The lemma now follows.  \qed\end{proof}

  \begin{proof1.4}
    We use induction on $l$. The base case where $l=1$ is trivially true. If
    $l>1$, let $y=x_2* x_3* \cdots* x_l$ so that by induction, it suffices to
    show
    \begin{equation}
      \label{eq:w1}
      t_x=t_{x_1}\cdot t_y.
    \end{equation}
By definition of dihedral segments, $x_1$ must be of the form
$\cdots st$ for some $s,t\in S$ and the start of $x_2$, hence $y$, must be of
the form $tu\cdots$ where $u\in S\setminus\{s,t\}$.

Recall the notations $s_k,t_k$ from Section \ref{sec:dihedral case}. Also
recall from the proof of Proposition \ref{dihedral KL mult} that when
$u,v$ are distinct elements of $\{s,t\}$, then we have $c_{uv}=c_uc_v$ and
$c_{u_{k+1}}=c_uc_{v_k}-c_{u_{k-1}}$. This allows us to prove Equation \eqref{eq:w1} by induction on the length $k=l(x_1)$ of
    $x_1$. First, if $k=2$, then Lemma
    \ref{truncation} implies that
    \[
      c_{x_1}c_{y}=c_{st}c_{tu\cdots}=c_sc_tc_{tu\cdots}=(v+v\inverse)c_{stu\cdots}=(v+v\inverse)c_{x_1*
      y}
    \]
    in $\calh_1$. This in turn implies Equation \eqref{eq:w1} by Corollary \ref{J_C shortcut}.
    Second, suppose $k>2$, write $x_1=s_1s_2s_3\cdots s_k$, and let $x_1'=s_2s_3\cdots
    s_k$ and
    $x_1''=s_3\cdots s_k$. Then Lemma
    \ref{truncation} implies that
    \[
      c_{s_1s_2}\cdot
      c_{x_1'}=c_{s_1}c_{s_2}c_{x_1'}=(v+v\inverse)c_{s_1}c_{x_1'}=(v+v\inverse)(c_{x_1}+c_{x_1''})
    \]
    and similarly
    \[
      c_{s_1s_2}\cdot
      c_{x_1'* y}=(v+v\inverse)(c_{x_1* y}+c_{x_1''* y}).
    \]
    The last two equations imply that
    \begin{eqnarray*}
      t_{s_1s_2}t_{x_1'}&=& t_{x_1}+t_{x_1''}\;,\\
      t_{s_1s_2}t_{x_1'\cdot y}&=& t_{x_1\cdot
      y}+t_{x_1''\cdot y}\;
    \end{eqnarray*}
    by Corollary \ref{J_C shortcut},
    therefore
    \[
      t_{x_1}t_y=(t_{s_1s_2}t_{x_1'}-t_{x_1''})t_y=t_{s_1s_2}t_{x_1'\cdot
      y}-t_{x_1''* y}=t_{x_1* y}+t_{x_1''* y}-t_{x_1''*
      y}=t_{x_1* y}=t_x.
    \]
   Here, the second equality holds by the inductive hypothesis because
    $l(x_1')<l(x_1)$. This completes our proof.
  \qed\end{proof1.4}

  \subsection{Products in $J_C$}
  \label{sec:product computation}
  Let $W$ be an arbitrary Coxeter system and let $C$ be its subregular cell.
 Equipped with the knowledge of $J_C$ in the dihedral case
 from Corollary \ref{d3} and with the reduction to the dihedral case
provided by Theorem \ref{dihedral factorization}, we are ready to
compute any product of the form $t_x\cdot t_y$ in $J_C$. We start with a simple case.
 \begin{proposition}
   \label{d1}
    Let $x,y\in C$. If the last letter of $x$ does not equal the first letter
    of $y$, then $t_xt_y=0$.
 \end{proposition}
 \begin{remark}
   Here we identify $x$ and $y$ with their reduced words. For example, by
   ``the last letter of $x$'' we mean the last letter of the unique reduced word of
   $x$. Since there is no ambiguity, we shall do so for all elements of
   $C$ from now on.
 \end{remark}
 \begin{proof}
   The assumptions imply that $x\not\sim_L y$ by Proposition
\ref{subregular}, therefore $t_xt_y=0$ by Part (2) of Proposition
\ref{gamma and cells}.
 \qed\end{proof}
 
  \begin{proposition}
    \label{d2}
    Let $x,y \in C$. Suppose the last letter of $x$ equals the first letter
    of $y$ but the last dihedral segment of $x$ and the first dihedral segment
    of $y$ involve different sets of letters. Then $t_xt_y=t_{x * y}$. 
  \end{proposition}
  \begin{proof}
    Let $x_1,\cdots, x_p$ and $y_1,\cdots, y_q$ be the dihedral segments of
    $x$ and $y$, respectively. By the assumptions, $x_1,\cdots, x_k, y_1,\cdots,
    y_l$ are exactly the dihedral segments of the glued product $x*y$, therefore
    Theorem \ref{dihedral factorization} implies 
    $
      t_xt_y=t_{x_1}\cdots t_{x_p} t_{y_1}\cdots t_{y_q}=t_{x_1 * \cdots
      * x_p * y_1 * \cdots * y_q}=t_{x * y}.$
\qed
\end{proof}

 It remains to consider the case where $x$ ends with the letter $y$ starts with
 and the last dihedral segment of $x$, say $x_p$, involves the same two letters
 as the the first dihedral segment, say $y_1$, of $y$. To compute $t_xt_y$ in
 this case, we need to use Theorem \ref{dihedral factorization} to 
 factor $t_x$ and $t_y$ and then compute $t_{x_p}\cdot t_{y_1}$ by Corollary
 \ref{d3}. Repeated use of this idea and Proposition \ref{d2} will eventually
 allow us to express $t_xt_y$ as a linear combination of $t_z (z\in C)$. We
 illustrate this below.

\begin{example}
  \label{iterated example}
  Suppose $S=\{1,2,3\}$, $m(1,2)=4, m(1,3)=5$ and $m(2,3)=6$.
  \begin{enumerate}[leftmargin=2em]
    \item Let $x=123, y=323213$. Then by Theorem \ref{dihedral
      factorization}
      and  Proposition \ref{d3}, 
      \begin{eqnarray*}
        t_xt_y&=& t_{12}t_{23}t_{3232}t_{21}t_{13}\\
        &=& t_{12}(t_{232}+t_{23232})t_{21}t_{13}\\
        &=& t_{12}t_{232}t_{21}t_{13}+t_{12}t_{23232}t_{21}t_{13}.
      \end{eqnarray*}
      Applying Theorem \ref{dihedral factorization} again to the last expression, we
      have
      \[
        t_{x}t_y=t_{123213}+t_{12323213}.
      \]
    \item Let $x=123, y=3213$. Repeated use of Theorem \ref{dihedral
      factorization} and
      Proposition \ref{d3} yields
      \begin{eqnarray*}
        t_xt_y&=& t_{12}t_{23}t_{32}t_{21}t_{13}\\
        &=& t_{12}(t_{2}+t_{232})t_{21}t_{13}\\
        &=& (t_{12}t_2) t_{21}t_{13}+t_{12}t_{232}t_{21}t_{13}\\
        &=& (t_{12}t_{21})t_{13}+t_{12}t_{232}t_{21}t_{13}\\
        &=& (t_{1}+t_{121})t_{13}+t_{12}t_{232}t_{21}t_{13}\\
        &=& t_1t_{13}+t_{121}t_{13}+t_{12}t_{232}t_{21}t_{13}\\
        &=& t_{13}+t_{1213}+t_{123213}.
      \end{eqnarray*}
      Note here that since $t_{12}t_2=t_{12}$ on the third line, the appearance
      of $t_2$ on the second line essentially means that after using
      Proposition \ref{d3} for the generators $2$ and 3 to obtain the second
      equality, we need to use the proposition again, now for the generators 1
      and 2, to carry on the computation.
  \end{enumerate}
\end{example}

We have now described how to compute $t_xt_y$ in all cases. Some Sage
(\cite{sagemath}) code implementing the computation is available at \cite{mycode}.

\subsection{Computation of $J_C$}
Now that we know how to compute the products of any two basis elements in
$J_C$, we wish to be able to efficiently enumerate all the basis elements. We
design a graph to achieve this now.

Keep Assumption  \ref{no rep}. Then by Proposition \ref{Matsumoto}, a word
is the reduced word of an element in $C$ if and only if none of its dihedral
segments is an $(s,t)$-braid for some distinct $s,t\in S$. In other words, to
write down the reduced word of an element in $C$, we only need to make sure not
to produce a dihedral segment that is ``too long'' in the process. This motivates the
following definition.
\begin{definition}[Subregular graph]
  \label{subregular graph}
  Let $H, T:S^*\setminus\{\emptyset\}\ra S$ be the functions that send any
  nonempty word $w=s_1s_2\cdots s_k$ to its first letter $s_1$ and last letter
  $s_k$, respectively.  For $s,t\in S$ and $k\in \Z_{\ge 1}$, let $(s,t)_k$ be
  the alternating word $sts\cdots$ of length $k$. Let $D=(V,E)$ be the directed graph such that
  \begin{enumerate} 
    \item $V=\{(s,t)_k: 
      s,t\in S, 0< k <m(s,t)\}$, 
    \item $E$ consists of directed edges $(v,w)$ pointing
      from $v$ to $w$, where 
      \begin{enumerate} 
        \item 
          either $v=(s,t)_{k-1}$ and $w=(s,t)_k$ for some $s,t\in S$, $1<k<
          m(s,t)$,
        \item or $v$ and $w$ are alternating words containing different sets of
          letters, with
          $T(v)=H(w)$.
      \end{enumerate} 
  \end{enumerate} 
  We call the graph $D$ the
  \emph{subregular graph} of $(W,S)$.  \end{definition}

Recall that a \emph{walk} on a directed graph is a sequence of vertices
$(v_1,v_2,\cdots, v_k)$ such that $(v_i,v_{i+1})$ is an edge for all $1\le i\le
k-1$. It is easy to check that walks on $D$ correspond bijectively to elements
of $C$ via the map $(v_1,\cdots ,v_k)\mapsto x=T(v_1)\cdots T(v_k)$, with the
vertices in the walk keeping track of the dihedral segments of $x$ as we write
down $x$ from left to right. We leave this as an exercise.

For any
$s\in S$, the elements of
$\Gamma_{s}\cap\Gamma_{s}\inverse$ correspond to
walks on the subregular graph $D$ that start at the vertex $v=(s)$
(an alternating word of length 1) and
end at a vertex $w$ with $T(w)=s$. We will call such a walk an
\emph{$s$-walk}. The collection of all $s$-walks often involve
only a proper subset $V'$ of the vertex set $V$ of $D$. We denote the subgraph of
$D$ induced by $V'$ by $D_s$ and call it the \emph{subregular $s$-graph}. We
remark that the construction of $D$ and $D_s (s\in S)$ is similar to that of
the graphs $\Gamma_s (s\in S)$ in Section 3.7 of \cite{subregular}.

An interesting feature of $D$ and $D_s (s\in S)$ is that when $m(s,t)<\infty$
for all $s,t\in S$, the vertex
  sets of $D$ and $D_s$ ($s\in S$) are finite, hence $D$ and $D_s$
  can be viewed as \emph{finite state automata} that recognize $C$ and
  $\Gamma_s\cap\Gamma_s\inverse$, respectively, in the sense of formal
  languages (see \cite{automata}).  

\begin{example}
  \label{odd example}
  Let $(W,S)$ be the Coxeter system whose Coxeter diagram is the triangle in
  Figure \ref{fig:1}. Then $D_1$ is the directed graph
  on the right, and elements of $\Gamma_1\cap\Gamma_1\inverse$ correspond to walks on
  $D_1$ that start with the top vertex and end with either the bottom-left or
  bottom-right vertex. 

  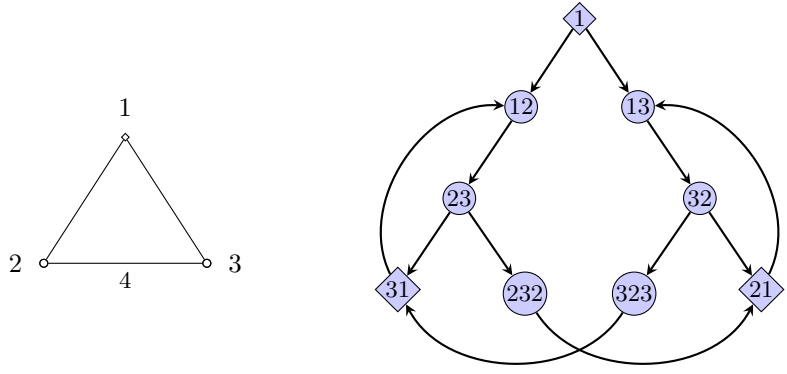
\begin{figure}[h!]
    \label{fig:1}
    \begin{center}
      \begin{minipage}{0.4\textwidth}
        \begin{tikzpicture}
          \node (00) {};
          \node (0) [right=1cm of 00] {};
          \node[main node] [diamond] (1) [right = 1cm of 0] {};
          \node[main node] (2) [below left = 1.6cm and 1cm of 1] {};
          \node[main node] (3) [below right = 1.6cm and 1cm of 1] {};
          \node (111) [above=0.1cm of 1] {$1$};
          \node[main node] (2) [below left = 1.6cm and 1cm of 1] {};
          \node (222) [left=0.1cm of 2] {$2$};
          \node[main node] (3) [below right = 1.6cm and 1cm of 1] {};
          \node (333) [right=0.1cm of 3] {$3$};
        %\node (13) [below right = 1.3cm and 6cm of 1] {$\leadsto$};
          \path[draw]
          (1) edge node {} (2)
          edge node {} (3)
          (2) edge node [below] {\small{4}} (3);
        \end{tikzpicture}
      \end{minipage}%
      \begin{minipage}{0.7\textwidth}
        \begin{tikzpicture}
          \node (000) {};
          \node[state node] [diamond] (4) [right=3cm of 000] {\small{$1$}}; 
          \node[state node] (5) [below left = 0.9cm and 0.5cm of 4]
          {\small{$12$}};
          \node[state node] (6) [below right = 0.9cm and 0.5cm of 4]
          {\small{$13$}};
          \node[state node] (7) [below left = 0.9cm and 0.5cm of 5]
          {\small{$23$}};
          \node[state node] (8) [below right = 0.9cm and 0.5cm of 6]
          {\small{$32$}};
          \node[state node] [diamond] (9) [below left = 0.9cm and 0.5cm of 7]
          {\small{31}};
          \node[state node] (10) [below right = 0.9cm and 0.5cm of 7]
          {\small{232}};
          \node[state node] (11) [below left = 0.9cm and 0.5cm of 8]
          {\small{323}};
          \node[state node] [diamond] (12) [below right = 0.9cm and 0.5cm of 8]
          {\small{21}};

          \path[draw,-stealth,,thick]
          (4) -- (5);

          \path[draw,-stealth,,thick]
          (4) -- (6); 

          \path[draw,-stealth,,thick]
          (5) -- (7); 

          \path[draw,-stealth,,thick]
          (6) -- (8);

          \path[draw,-stealth,,thick] 
          (7) -- (9); 

          \path[draw,-stealth,,thick]
          (7) -- (10);

          \path[draw,-stealth,,thick]
          (8) -- (11);

          \path[draw,-stealth,,thick]
          (8) -- (12);

          \path[draw,thick]
          (9) edge [bend left=60,-stealth,thick] (5);
          \path[draw,semithick]
          (12) edge [bend right=60,-stealth,thick] (6);
          \path[draw,semithick]
          (10) edge [bend right=60,-stealth,thick] (12);
          \path[draw,semithick]
          (11) edge [bend left=60,-stealth,thick] (9);
        \end{tikzpicture}
      \end{minipage}
    \end{center}
    \caption{The Coxeter diagram and subregular $1$-graph of $(W,S)$}
  \end{figure}
\end{example}

By using the subregular graph to keep track of the basis elements
and using Section \ref{sec:product computation} to compute their products, we may now
work out the entire multiplication tables of $J_C$ and $J_s$. We include an
example below. 

\begin{example}
  \label{counter example}
Let $(W,S)$ be the Coxeter system with the
  following diagram. We will compute the algebra $J_1$.
  
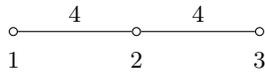
\begin{figure}[h!]
    \centering
    \begin{tikzpicture}
      \node[main node] (1) {};
      \node (11) [below=0.1cm of 1] {\small{1}};
      \node[main node] (2) [right=1.5cm of 1] {};
\node (22) [below=0.1cm of 2] {\small{2}};
      \node[main node] (3) [right=1.5cm of 2] {}; 
      \node (33) [below=0.1cm of 3] {\small{3}};
      \path[draw]
      (1) edge node [above] {\small{4}} (2)
      (2) edge node [above] {\small{4}} (3);
    \end{tikzpicture}
    \caption{The Coxeter diagram of $(W,S)$}
  \end{figure}
\end{example}

The graph $D_1$ is shown in Figure \ref{fig:D1}. Let $x={121}$, $y=12321$, and let $y_n$ denote the glued product
$y*y*\cdots *y$ of $n$ copies of $y$ for each $n\in \Z_{\ge 1}$. Then by using $D_1$,
it is easy to see that $\Gamma_1\cap\Gamma_1\inverse$ consists exactly of 1,
$x$ and all $y_n$ where
$n\ge 1$, hence the basis elements of $J_1$ are $t_1,t_x$ and
$t_n:=t_{y_n}$ where $n\ge 1$.

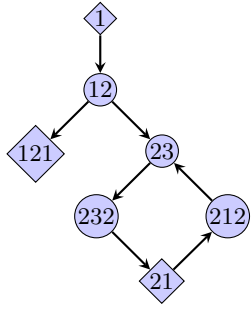
\begin{figure}[h!]
  \centering
%\begin{center}
  \begin{tikzpicture}
          \node[state node] [diamond] (1) {\small{$1$}}; 
          \node[state node] (2) [below = 0.5cm of 1]
          {\small{$12$}};
          \node[state node] [diamond] (3) [below left = 0.5cm and 0.5cm of 2]
          {\small{$121$}};
          \node[state node] (4) [below right = 0.5cm and 0.5cm of 2]
          {\small{$23$}};
          \node[state node] (5) [below left = 0.5cm and 0.5cm of 4]
          {\small{$232$}};
          \node[state node] [diamond] (6) [below right = 0.5cm and 0.5cm of 5]
          {\small{21}};
          \node[state node] (7) [below right = 0.5cm and 0.5cm of 4]
          {\small{212}};

          \path[draw,-stealth,,thick]
          (1) -- (2);

          \path[draw,-stealth,,thick]
          (2) -- (3); 

          \path[draw,-stealth,,thick]
          (2) -- (4); 

          \path[draw,-stealth,,thick]
          (4) -- (5);

          \path[draw,-stealth,,thick] 
          (5) -- (6); 

          \path[draw,-stealth,,thick]
          (6) -- (7);

          \path[draw,-stealth,,thick]
          (7) -- (4);
  \end{tikzpicture}
%\end{center}
  \caption{The subregular $1$-graph of $(W,S)$}
  \label{fig:D1}
\end{figure}

Recall that $t_1t_w=t_w=t_wt_1$ as $t_1$ is the unit of $J_1$. Next, note that propositions
\ref{d2} and \ref{d3} imply that $t_xt_x=t_{121}t_{121}=t_1$,
while
\begin{equation}
  \label{eq:check n}
  t_{x}t_{n}=t_{121}t_{12321\cdots}=t_{121}t_{12}t_{2321*y_{n-1}}=t_{12}t_{2321*y_{n-1}}=t_{n}
\end{equation}
where we set $y_0=1$. Similarly, $t_nt_x=t_n$ for all $n\ge 1$. Finally, if we
set $t_0=t_1+t_x$, then by computations similar to those in Equation
\eqref{eq:check n}, we can check that $t_1t_n=t_{n-1}+t_{n+1}$ for all $n\ge
1$. By induction, we can then show that
\begin{equation}
  \label{eq:44}
  t_mt_n=t_{\abs{m-n}}+t_{m+n}
\end{equation}
for all $m,n\ge 1$. We have now computed all products $t_xt_y$ where $x,y\in
\Gamma_{1}\cap\Gamma_s\inverse$. To summarize, $J_1$ has the following multiplication
table, where the entry in the row labeled by $A$ and the column labeled by
$B$ has value $AB$. Note that $J_1$ is in fact commutative, which is not
obvious from its definition.
\begin{table}[h!]
\caption{The multiplication table of $J_1$}
\begin{center}
\begin{tabular}[h!]{c|ccc}
 & $t_1$ & $t_x$ & $t_n$\\
 \hline
 $t_1$ & $t_1$ & $t_x$ & $t_n$\\
 $t_x$ & $t_x$ & $t_1$ & $t_n$\\
 $t_m$ & $t_m$ & $t_m$ & $t_{\abs{m-n}}+t_{m+n}$
\end{tabular}
\end{center}
\vspace{1em}
\end{table}

\section{$J_C$ and the Coxeter diagram} 
\label{sec:diagram}
Let $(W,S)$ be an arbitrary Coxeter
system, let $G$ be its Coxeter diagram, and let $J_C$ be its subregular $J$-ring.
We study the relationship between $J_C$ and $G$ in
this section.

\subsection{Simply-laced Coxeter systems} 
\label{sec:simply-laced}
We assume $(W,S)$ is simply-laced and prove Theorem \ref{simply-laced} in this subsection. Recall
that this means $m(s,t)=3$ for all $s,t\in S$ that are adjacent in $G$. We may hence think of $G$ as an unweighted undirected graph.

As in the directed case, we define a \emph{walk} on $G$ is to be a sequence $P=(v_1,\cdots ,v_k)$ of vertices in $G$
such that $\{v_i,v_{i+1}\}$ is an edge for all $1\le i\le
k-1$. We define a \emph{spur} on
$G$ to be a walk of the form $(v,v',v)$ where $\{v,v'\}$ forms an edge. 
The following observation is key to this subsection.

\begin{proposition}
  \label{simply-laced bijection}
  For each $x\in C$, let
  $P_x:=(s_1,\cdots,s_k)$ where $s_1\cdots s_k$ is the reduced word of $x$. Then the map $x\mapsto P_x$ is a bijection between
  $C$ and the set of walks on $G$ with no spurs.
\end{proposition}
\begin{proof}
We first show that for any $x\in C$, $P_x$ is a walk on $G$ with no spurs. Since
$x$ has a unique reduced word, $x_i$ and $x_{i+1}$ cannot commute, therefore
$m(x_i,x_{i+1})=3$ for all $1\le i\le k-1$ and $P_x$ is a walk.
The fact that $m(x_i,x_{i+1})=3$ means that $x$ cannot contain the
$(x_i,x_{i+1})$-braid $(x_i,x_{i+1},x_i)$, therefore $P_x$ has no spur.

Conversely, if $(s_1,\cdots, s_k)$ is a walk on $G$ with no spur, then the word
$w=s_1\cdots s_k$ must be the reduced word of an element in $C$. Indeed, in this
case the
dihedral segments of $w$ are exactly $x_ix_{i+1}$ for $1\le i\le
k-1$, and $m(s_i,s_{i+1})=3$ for each $i$. This implies that $w$ contains no
$(s,t)$-braid for any $s,t\in S$, therefore $w$ is the reduced word of an
element in $C$ by Proposition \ref{Matsumoto}. 

The previous paragraph implies that the map $x\mapsto P_x$ is surjective. Since it is
clearly injective, it is a bijection.
\qed\end{proof}

Before we prove Theorem \ref{simply-laced}, let us define the
fundamental groupoid of $G$. Recall that 
%a \emph{groupoid} may be
%viewed as a generalization of a group, in that it is defined to be a pair
%$(\mathcal{G},\circ)$ where $\mathcal{G}$ is set and  $\circ$ is a partially-defined binary
%operation on $\mathcal{G}$ that satisfy certain
%axioms (see \cite{groupoid-def}). In partcular, 
for any topological space
$X$ and a subset $A$ of $X$, the \emph{fundamental groupoid of
$X$ based on $A$} is defined to be $\Pi(X,A):=(\mathcal{P},\circ)$, where
$\mathcal{P}$ are the homotopy equivalence classes of paths on $X$ that
connect points in $A$ and $\circ$ is concatenation of paths. Now, we may view $G$ as embedded in a topological
surface and hence as a topological space with the subspace topology. This allows us to define the \emph{fundamental groupoid of $G$} to be
$\Pi(G):=\Pi(G,S)=(\mathcal{P},\circ)$, where $\mathcal{P}$ stands for paths on
$G$. 

Clearly, paths on $G$ are just walks,
and concatenation of paths correspond to concatenation of walks. Here, more precisely, for any two walks $P=(v_1,\cdots,v_{k-1}, v_k)$ and $Q=(u_1, u_2\cdots
u_l)$ on $G$, by their \emph{concatenation} we mean the 
walk $P\circ Q = (v_1,\cdots, v_{k-1}, v_k, u_2,\cdots, u_l)$ if $v_k=u_1$;
otherwise we leave $P\circ Q$ undefined. 

We now elaborate on the notion of homotopy equivalence of walks on $G$. Given any walk on $G$ containing a spur, i.e., a walk of the form $P_1=(\cdots, u, v,v',v,
u', \cdots)$, we
may {remove the spur} to form a new walk $P_2=(\cdots, u, v, u', \cdots)$;
conversely, we can insert a spur $(v,v',v)$ to a walk of the form
$P_2$ to obtain the walk $P_1$. Note that two walks on $G$ are
homotopy equivalent if and only if they can be obtained from each other by
a sequence of removals or insertion of spurs. Further, each homotopy
equivalence class of walks contains a unique walk with no spurs.  Consequently,
let $[P]$ denote the homotopy equivalence class of a walk $P$, then
the bijection $x\mapsto P_x$ from Proposition \ref{simply-laced bijection}
induces a unique $\Z$-module
isomorphism $\Phi: J_C \ra \Z\Pi(G)$ defined by   
\begin{equation}
  \label{vs iso}
  \Phi(t_{x})=[P_x],\qquad
  \forall x\in C.
\end{equation}
We will show that $\Phi$ is actually an isomorphism of based rings 
in Theorem \ref{simply-laced}.

For each vertex $s$ in $G$, we define the \emph{fundamental
group of $G$ based at $s$} to be $\Pi_s(G)=(\mathcal{P}_s,\circ)$, where
$\mathcal{P}_s$ are now equivalence classes of walks on $G$ that start and end
with $s$, and $\circ$ is concatenation as before. Of course, $\Pi_s(G)$ is
actually a group, so it makes
sense to talk about the its group algebra $\Z\Pi_s(G)$ over $\Z$. We
mimic the construction of a group algebra to define a natural counterpart of $\Z\Pi_s(G)$ for $\Pi(G)$:

\begin{definition}
  \label{groupoid mult}
  Let $\Pi(G)=(\mathcal{P},\circ)$ be the fundamental groupoid of $G$.
  We define the \emph{groupoid algebra} of $\Pi({G})$ over $\Z$ to be the
  free abelian group $\Z\mathcal{P}=\oplus_{[P]\in \mathcal{P}} \Z[P]$ equipped with an
  $\Z$-bilinear multiplication $\cdot$ where
  \[
    [P]\cdot [Q]= 
    \begin{cases}
      [P\circ Q]\quad&\text{if}\;P\circ Q \;\text{is defined in $G$},\;
      \\
      0\quad&\text{if}\; P\circ Q\;\text{is not defined}.
    \end{cases}
  \]
\end{definition}
\noindent Note that $\Z\Pi(G)$ is clearly associative. 

The rings $\Z\Pi(G)$ and $\Z\Pi_s(G)$ are naturally based rings in the
following way:
\begin{proposition}
  \label{groupoid base}
Define $P_s$ to be the constant walk $(s)$ for
  all $s\in S$.
  For each walk $P=(v_1,v_2,\cdots,v_k)$ on $G$, define
  $P\inverse=(v_k,\cdots,v_2,v_1)$.  Then 
  \begin{enumerate}
    \item  the groupoid algebra $\Z\Pi(G)$ is a based ring with basis
  $\{[P]\}_{[P]\in \mathcal{P}}$, with unit $\sum_{s\in S} [P_s]$ (so the
  distinguished index set simply corresponds to $S$), and with its
  anti-involution induced by the map $[P]\mapsto [P\inverse]$;

\item for each
  $s\in S$, the group algebra $\Z\Pi_s(G)$ is a unital
  based ring with basis $\{[P]\}_{[P]\in \mathcal{P}_s}$, with unit $[P_s]$ 
  (so the distinguished index set is simply $\{(s)\}$), and with its
  anti-involution induced by the map $[P]\mapsto [P\inverse]$.
  \end{enumerate}
\end{proposition}
\begin{proof}
  All the claims are easy to check using definitions.
\qed\end{proof}

We are ready to prove Theorem \ref{simply-laced}, which is restated in
abbreviated form below.
\begin{thm1.1}
We have $J_C\cong \Z\Pi(G)$ as based rings, and
  $J_s\cong \Z\Pi_s(G)$ as based rings for all $s\in S$.  
\end{thm1.1}

\begin{proof}
  We show that the $\Z$-module isomorphism $\Phi: J_C\ra \Z\Pi(G)$ defined by
  Equation \ref{vs iso} is an algebra
  homomorphism. This would imply $J_s\cong \Z\Pi_s(G)$ for all $s\in S$, since
  $\Phi$ clearly restricts to a $\Z$-module map from $J_s$ to $\Z \Pi_s(G)$.
  The fact that $\Phi$ and the restrictions are actually isomorphisms of based
  rings will then be clear once we compare the based ring
  structure of $J_C, \Z\Pi(G), J_s$ and $\Z\Pi_s(G)$ described in Corollary
  \ref{subregular base} and
  Proposition \ref{groupoid base}.

  To show $\Phi$ is an algebra homomorphism, we need to show
  \begin{equation}
    \label{eq:groupoid hom}
    [P_x]\cdot [P_y]=\Phi(t_xt_y)
  \end{equation}
  for all $x,y\in C$. Let $s_k\cdots s_1$ and
  $u_1\cdots u_l$ be the reduced word of $x$ and $y$, respectively. If
  $s_1\neq u_1$, then Equation \eqref{eq:groupoid hom} holds since both sides are
  zero by Definition \ref{groupoid mult} and Proposition \ref{d1}. If
  $s_1=u_1$, let $q\le \min(k,l)$ be the largest integer such that $s_i=u_i$
  for all $1\le i\le q$. Then
  \begin{eqnarray*}
    [P_x]\cdot [P_y]&=& [(s_k,\cdots, s_{q+1}, s_q, \cdots, s_1)\circ (s_1,\cdots,
    s_q,u_{q+1},\cdots, u_l)]\\
    &=& [(s_k,\cdots, s_{q+1}, s_q,\cdots, s_2,s_1, s_2,\cdots,
    s_q,u_{q+1},\cdots, u_l)]\\
    &=& [(s_k,\cdots, s_{q+1},s_q,u_{q+1},\cdots, u_l)],
  \end{eqnarray*}
  where the last equality holds by successive removal of spurs of the form
  $(s_{i+1},s_i,s_{i+1})$. Meanwhile,  for each
  $1\le i\le q$, since $m(s_i,s_{i+1})=3$, Proposition \ref{d3} implies that
  \begin{equation}
    \label{eq:algebraic spur}
    t_{s_{i+1}s_{i}}t_{s_is_{i+1}}=t_{s_{i+1}}.  \end{equation}
  By calculations like those in Example \ref{iterated example}, it is then
  straightforward to check
  that 
   \begin{eqnarray*}
    %t_xt_y&=& (t_{s_k\cdots s_{q+1}s_q} t_{s_qs_{q-1}}\cdots
    %t_{s_3s_2}t_{s_2s_1})(t_{s_1s_2}t_{s_2s_3}\cdots
    %t_{s_{q-1}s_q}t_{s_qu_{q+1}\cdots
    %u_{l}})\\
    %&=& (t_{s_k\cdots s_{q+1}s_q} t_{s_qs_{q-1}}\cdots
    %t_{s_3s_2})t_{s_2}(t_{s_2s_3}\cdots t_{s_{q-1}s_q}t_{s_qu_{q+1}\cdots
    %u_{l}})\\
    %&=& (t_{s_k\cdots s_{q+1}s_q} t_{s_qs_{q-1}}\cdots
    %t_{s_3s_2})(t_{s_2s_3}\cdots t_{s_{q-1}s_q}t_{s_qu_{q+1}\cdots
    %u_{l}})\\
    %%&=& (t_{s_k\cdots s_{q+1}s_q} t_{s_qs_{q-1}}\cdots
    %%t_{s_4s_3})t_{s_3}(t_{s_3s_4}\cdots t_{s_{q-1}s_q}t_{s_qu_{q+1}\cdots
  %%u_{l}})\\
    %%&=& (t_{s_k\cdots s_{q+1}s_q} t_{s_qs_{q-1}}\cdots
    %%t_{s_4s_3})(t_{s_3s_4}\cdots t_{s_{q-1}s_q}t_{s_qu_{q+1}\cdots
  %%u_{l}})\\
    %&=& \cdots\\
    %&=& t_{s_k\cdots s_{q+1}s_q}t_{s_qu_{q+1}\cdots u_{l}}\\
t_xt_y    &=& t_{s_k\cdots s_{q+1}s_qu_{q+1}\cdots s_l'}.
  \end{eqnarray*}
    By the definition of $\Phi$, this implies that
  \[
    \Phi(t_xt_y)=[(s_k,\cdots, s_{q+1},s_q, u_{q+1},\cdots, u_l)].
  \]
  and hence 
  $
  [P_x]\cdot [P_y]=\Phi(t_xt_y).
  $ 
  Our proof is now complete.
\qed\end{proof}

\subsection{Oddly-connected Coxeter systems}
\label{sec:oddly-connected def}
Define a Coxeter system
$(W,S)$ to be \emph{oddly-connected} if for all distinct $s,t\in S$, there is a
walk in the Coxeter diagram $G$ of the form
$(s=v_1,v_2,\cdots, v_k=t)$ where the edge weight $m(v_{i},v_{i+1})$ is odd for all $1\le i\le k-1$. In this subsection,
we discuss how the odd-weight edges affect the structure of the algebras
$J_C$ and $J_s$ ($s\in S$).

We need some relatively heavy notation.
\begin{definition}
  \label{transition}
  For any $s,t\in S$ such that $M=m(s,t)$ is odd, 
  \begin{enumerate}
    \item we define  
      \[
        z({st})=sts\cdots t
      \]
      to be the alternating word of length $M-1$ that starts with
      $s$ (note that $z(st)$ necessarily ends with $t$ now that $M$ is odd);
    \item we define maps $\lambda_{s}^t, \rho_t^s: J_C\ra J_C$ by
      \[
      \lambda_{s}^t(t_x)=t_{z({ts})}t_{x},\]
      \[
        \rho_s^t
        (t_x)=t_xt_{z({st})},
      \]
      and define the map $\phi_{s}^t: J_C\mapsto J_C$ by
      \[
        \phi_{s}^t(t_x)=\rho_{s}^t\circ\lambda_{s}^t (t_x)
      \]
      for all $x\in C$.
    \end{enumerate}
\end{definition}

\begin{remark}
  \label{match}
  The notation above is set up in the following way. The letters $\lambda$
  and $\rho$ indicate a map is multiplying its input by an element on the left and right, respectively. The subscripts and superscripts are to provide mnemonics for what the maps do on the reduced words indexing the basis elements of
  $J_C$: note that by Proposition \ref{subregular} and Part (3) of Proposition
  \ref{gamma and cells}, 
  $\lambda_{s}^t$ maps $J_{\Gamma_s^{-1}}$ to
  $J_{\Gamma_{t}^{-1}}$ and vanishes on $J_{\Gamma_h^{-1}}$ for any $h\in
  S\setminus\{s\}$. 
  Similarly, $\rho_{s}^t$ maps $J_{\Gamma_s}$ to
  $J_{\Gamma_{t}}$ and vanishes on $J_{\Gamma_h}$ for any $h\in S\setminus\{s\}$. 
\end{remark}

\begin{proposition}
  \label{odd edge}
  Let $s,t$ be as in Definition \ref{transition}. Then
  \begin{enumerate}
    \item 
      $\rho_s^t\circ \lambda_s^t=\lambda_s^t\circ\rho_s^t$. 
    \item $\rho_{t}^s\circ\rho_{s}^t(t_x)=t_x$ for any $x\in\Gamma_s$,
      $\lambda_t^s\circ\lambda_{s}^t(t_x)=t_x$ for any $x\in\Gamma_s^{-1}$.
    \item  
      $\rho_{s}^t (t_x)\lambda_{s}^t(t_y)=t_xt_y$ for any $x\in \Gamma_s,y\in
      \Gamma_s^{-1}$.
    \item The restriction of
      $\phi_s^t$ on $J_s$ is an isomorphism of based rings from $J_s$ to $J_t$.

  \end{enumerate}
\end{proposition}

\begin{proof}
  Part (1) holds since both sides of the equation sends $t_x$ to
  $t_{z(ts)}t_xt_{z(st)}$. Parts (2) and (3) are consequences of the truncated
  Clebsch--Gordan rule. By the rule, 
  \[
    t_{z(st)}t_{z(ts)}=t_s, 
  \]
  therefore $
  \rho_{t}^s\circ\rho_{s}^t(t_x)=t_xt_s=t_x
  $
  for any $x\in\Gamma_s$ and
  $
  \lambda_t^s\circ\lambda_{s}^t(t_x)=t_st_x
  $ 
  for any $x\in\Gamma_s^{-1}$; this proves (2). Meanwhile, $ 
  \rho_{s}^t
  (t_x)\lambda_{s}^t(t_y)=t_xt_{z(st)}t_{z(ts)}t_y=t_xt_st_y=t_xt_y
  $ for any $x\in\Gamma_s, y\in \Gamma_{s}^{-1}$; this proves (3).

  For part (4), the fact that $\phi_s^t$ maps $J_s$ to $J_t$ follows from Remark
  \ref{match}. To see that $\phi_{s}^t$ is a (unit-preserving) algebra homomorphism, note that
    \[
    \phi_s^t(t_s)=t_{z(ts)}t_{s}t_{z(st)}=t_{z(ts)}t_{z(st)}=t_t,
  \]
  and that 
  for all $t_x, t_y\in
  J_s$, 
  \[
    \phi_s^t
    (t_x)\phi_s^t(t_y)=(\rho_s^t (\lambda_s^t(t_x))\cdot
    (\lambda_s^t(\rho_s^t(t_y))=\lambda_s^t(t_x)\cdot
    \rho_s^t(t_y) 
    =\phi_{s}^t(t_xt_y)
  \]
  by parts (1) and (3).
  We can similarly check $\phi_t^s$ is an algebra homomorphism from $J_t$ to
  $J_s$. Finally, using
  calculations similar to those used for part (2), it
  is easy to check that $\phi_s^t$ and $\phi_{t}^s$ are mutual inverses , therefore $\phi_s^t$ is an algebra
  isomorphism.
  
  It remains to check that the restriction is an isomorphism of
  based rings. In light of Proposition \ref{subregular base}, this means
  checking that 
  $\phi_s^t(t_{x\inverse})=(\phi_s^t(t_x))^*$ for each $t_x\in J_s$, where ${}^*$ is the linear map
  sending $t_x$ to $t_{x\inverse}$ for each $t_x\in J_s$. This holds because 
  \[
    \phi_s^t(t_{x\inverse})=t_{z(ts)}t_{x\inverse}t_{z(st)}=(t_{z(st)\inverse}t_{x}t_{z(ts)\inverse})^*=(t_{z(ts)}t_xt_{z(st)})^*=(\phi_s^t(t_x))^*,
  \]
  where the second equality follows from the definition of ${}^*$ and the fact that $t_x\mapsto t_{x\inverse}$ defines an
  anti-homomorphism in $J$. 
\qed\end{proof}

Now we upgrade the definitions and Propositions from a single edge to a walk.
\begin{definition}
  \label{lambdas and rhos}
  For any walk $P=(u_1,\cdots, u_l)$ in $G$ where $m(u_k,u_{k+1})$ is odd
  for all $1\le k\le l-1$, we define maps $\lambda_{P},\rho_P$ by
      \[
        \lambda_P=\lambda_{u_{l-1}}^{u_l}\circ \cdots \lambda_{u_2}^{u_3}\circ
        \lambda_{u_1}^{u_2},
      \]
      \[
        \rho_P=\rho_{u_{l-1}}^{u_l}\circ \cdots \rho_{u_2}^{u_3}\circ
        \rho_{u_1}^{u_2},
      \]
      and define the map $\phi_P: J_C\ra J_C$ by 
      \[
        \phi_P=\lambda_P\circ\rho_P.
      \]
\end{definition}

\begin{proposition}
  \label{odd walk}
  Let $P=(u_1,\cdots, u_l)$ be as in Definition \ref{lambdas and rhos} Then
  \begin{enumerate}
    \item $ \phi_P=\phi_{u_{l-1}}^{u_l}\circ\cdots\circ \phi_{u_2}^{u_3}\circ
      \phi_{u_1}^{u_2}$.
    \item $\rho_{P\inverse}\circ\rho_{P}(t_x)=t_x$ for any
      $x\in\Gamma_{u_1}$,
      $\lambda_{P\inverse}\circ\lambda_P(t_x)=t_x$ for any
      $x\in\Gamma_{u_1}^{-1}$.
    \item $\rho_{P} (t_x)\lambda_P(t_y)=t_{x}t_y$ for any $x\in
      \Gamma_{u_1}, y\in \Gamma_{u_l}^{-1}$.
    \item The restriction of $\phi_P$ is an isomorphism of based rings from
      $J_{u_1}$ to $J_{u_l}$.
  \end{enumerate}
\end{proposition}
\begin{proof}
  Part (1) holds since each left multiplication $\lambda_{u_k}^{u_{k+1}}$
  commutes with all right multiplications $\rho_{u_{k'}}^{u_{k'+1}}$. Part
  (2)-(4) can be proved by writing out each of the maps as a composition of
  $(l-1)$ appropriate maps corresponding to the $(l-1)$ edges of $P$ and then repeatedly
  applying their counterparts in Proposition \ref{odd walk} on these
  component maps. In particular, (4) follows from (1) since a composition of
  isomorphisms of based rings is clearly another isomorphism of based rings.
\qed\end{proof}

We are almost ready to prove Theorem \ref{oddly-connected}: 
\begin{thm1.2}
  Let $(W,S)$ be an oddly-connected Coxeter system. Then 
  \begin{enumerate}
    \item $J_s\cong J_t$ as based rings for all $s,t\in S$. 
    \item $J_C\cong
      \mathrm{Mat}_{S\times S}(J_s)$ as based rings for all $s\in S$. In particular,
      $J_C$ is Morita equivalent to $J_s$ for all $s\in S$.
  \end{enumerate}
\end{thm1.2}
\noindent Here, for each $s\in S$, $\mathrm{Mat}_{S\times
S}(J_s)$ is the algebra of matrices with rows and columns indexed by $S$
and with entries from $J_s$. We explain its based ring structure below.

  \begin{proposition}
    \label{odd base}
For any $a,b\in S$ and $f\in J_s$, let
$E_{a,b}(f)$ be the matrix in $\mathrm{Mat}_{S\times S} (J_s)$ with $f$ at the
$a$-row, $b$-column and zeros elsewhere. Then 
  $\mathrm{Mat}_{S\times S}(J_s)$ is a based ring with basis
  $\{E_{a,b}(t_x):a,b\in S, x\in \Gamma_s\cap\Gamma_s\inverse\}$, with
  unit element $\sum_{s\in S}E_{s,s}(t_s)$, and with its anti-involution
  induced by $E_{a,b}(t_{x})^*=E_{b,a}(t_{x\inverse})$.  
\end{proposition} 

\begin{proof}
 Note that for any $a,b,c,d\in S$ and $f,g\in J_s$, 
 \begin{equation}\label{eq:matrix mult}
   E_{a,b}(f)E_{c,d}(g)=\delta_{b,c}E_{a,d}(fg).
 \end{equation}
It is then easy to check that $\mathrm{Mat}_{S\times S}(J_s)$ has unit
$\sum_{s\in S}E_{s,s}(t_s)$. Next, note that 
\[
  (E_{a,b}(f)E_{c,d}(g))^*=0=(E_{c,d}(g))^*(E_{a,b}(f))^*
\]
when $b\neq c$. When $b=c$, since $t_x\mapsto t_{x\inverse}$ is an
anti-homomorphism on $J$,
  \[
    (E_{a,b}(t_x)E_{c,d}(t_y))^*= (E_{a,d}((t_xt_y)))^*=
      E_{d,a}(t_{y\inverse}t_{x\inverse})=(E_{c,d}(t_y))^*(E_{a,b}(t_x))^*.
  \]
The last two equations imply that
  ${}^*$ induces an anti-involution of $\mathrm{Mat}_{S\times S}(J_s)$. Finally,
  note that $E_{u,u}(t_s)$ appears in
  $E_{a,b}(t_x)E_{c,d}(t_y)=\delta_{b,c}E_{a,d}(t_xt_y)$ for some $u\in S$ if and only if
  $b=c, a=d=u$ and $x=y\inverse$ (for $t_s$ appears in $t_xt_y$ if and only if
  $x=y\inverse$). This proves Equation \eqref{eq:based ring},
 completing all  
  necessary verifications. 
\qed\end{proof}

\begin{proof1.2}
  Part (1) follows from the last part of Proposition \ref{odd walk}. 
  To prove (2), fix $s\in S$. For each $t\in S$, fix a walk
  $P_{st}=(s=u_1,\cdots, u_l=t)$ and define $P_{ts}=P_{st}\inverse$. Write
  $\lambda_{st}$ for $\lambda_{P_{st}}$, and define
  $\rho_{st},\lambda_{ts},\rho_{ts}$ similarly. 
  Consider the
  unique $\Z$-module map 
  \[
    \Psi: J_C\ra \mathrm{Mat}_{S\times S}(J_s)
  \]
  defined as follows: 
  for any $t_x\in J_C$, say $x\in\Gamma_a^{-1}\cap\Gamma_b$ for $a,b\in S$, let
  \[
    \Psi(t_x)=E_{a,b}(\lambda_{{a s}}\circ \rho_{{bs}} (t_x)).
  \]
  We first show below that $\Psi$ is an algebra isomorphism. 

  Let $t_{x},t_y\in J_C$. Suppose
  $x\in\Gamma_a^{-1}\cap\Gamma_b$ and $y\in\Gamma_c^{-1}\cap\Gamma_d$ for
  $a,b,c,d\in S$. If $b\neq c$, then
  \[
    \Psi(t_x)\Psi(t_y)=0=\Psi(t_xt_y)
  \]
  by Equation \eqref{eq:matrix mult} and Proposition \ref{d1}. If $b=c$, then 
  \begin{eqnarray*}
    \Psi(t_x)\Psi(t_y)&=& E_{a,b}(\lambda_{{as}}\circ\rho_{bs}(t_x)) \cdot
    E_{c,d}(\lambda_{{cs}}\circ\rho_{ds}(t_y))\\
    &=& E_{a,d}([\lambda_{{as}}\circ\rho_{bs}(t_x)] \cdot
    [\lambda_{{bs}}\circ\rho_{ds}(t_y)])\\
    &=& E_{a,d}( (\lambda_{as}\circ\rho_{ds})  [\rho_{bs}(t_x)\cdot
    \lambda_{bs}(t_y)])\\
    &=& E_{a,d}( (\lambda_{as}\circ\rho_{ds})  [t_xt_y]) \\
    &=& \Psi(t_xt_y),
  \end{eqnarray*}
  where the second last equality holds by part (3) of Proposition \ref{odd walk}.
  It follows that $\Psi$ is an algebra homomorphism. Next, consider the map
  \[
    \Psi':\mathrm{Mat}_{S\times S}(J_s)\ra J_C
  \]
  defined by 
  \[
    \Psi'(E_{a,b}(f))=\lambda_{sa}\circ\rho_{sb}(f)
  \]
  for all $a,b\in S$ and $f\in J_s$. Using Part (2) of Proposition \ref{odd
  walk}, it is easy to check that $\Psi$ and $\Psi'$ are mutual inverses as maps
  of sets. It follows that $\Psi$ is an algebra isomorphism. Finally, it is
  easy to compare Proposition \ref{subregular base} with Proposition \ref{odd
  base} and check that $\Psi$ is indeed an isomorphism of based rings.\qed\end{proof1.2}

\begin{remark}
  The conclusions of the theorem fail in general when $(W,S)$ is not
  oddly-connected. As a counter-example, consider rings $J_1$ and
  $J_2$ arising from the Coxeter system in Example
  \ref{counter example}.  
By the truncated Clebsch--Gordan rule,
\[
  t_{212}t_{212}=t_2=t_{232}t_{232},
\]
therefore $J_2$ contains at least two basis elements with multiplicative order
2. However, it is evident from Example \ref{counter example} that $t_{121}$ is
the only basis element of order 2 in $J_1$. This implies that $J_1$ and $J_2$ are not
isomorphic as based rings. Moreover, Equation \eqref{eq:matrix mult} implies
that for any
$s\in S$, the basis elements of $\mathrm{Mat}_{S\times S}(J_s)$ of order 2 must be of the
form $E_{u,u}(t_{x})$ where $u\in S$ and $t_x$ is a basis element of order 2 in $J_s$, so $\mathrm{Mat}_{S\times S}(J_1)$ and $\mathrm{Mat}_{S\times
S}(J_2)$ have different numbers of basis elements of order 2 as well. It follows that
Part (2) of the theorem also fails.
\end{remark}

\begin{remark}
  The isomorphism between $J_s$ and $J_t$ can be easily lifted to a tensor
  equivalence between their categorifications $\mathcal{J}_s$ and
  $\mathcal{J}_t$, the subcategories of the category $\mathcal{J}$ mentioned in the introduction
  that correspond to $\Gamma_s\cap\Gamma_s^{-1}$ and
  $\Gamma_t\cap\Gamma_t^{-1}$.
\end{remark}

Let us end the section by revisiting an earlier example.

\begin{example}
  Let $(W,S)$ be the Coxeter system from Example \ref{odd example}. Clearly, $(W,S)$ is \emph{oddly-connected}, hence $J_3\cong J_2\cong
  J_1$
  and $J_C\cong \mathrm{Mat}_{3\times 3}(J_1)$ by Theorem
  \ref{oddly-connected}. Let us study $J_1$.
 Observe from $D_1$ that all walks corresponding to elements of
$\Gamma_1\cap\Gamma_1^{-1}$  can be obtained by
  concatenating the walks corresponding to the elements
  $x=1231, y=1321, z=12321$ and $w=13231$. By Section \ref{sec:product computation},
  this implies that $t_x,t_y,t_z,t_w$
  generate $J_1$. Computing the
  products of these elements reveals that $J_1$ is generated by $t_x, t_y, t_z, t_w$ subject to the following
  six relations:
  \[
    t_xt_y=1+t_z, \,t_yt_x=1+t_w,\, t_xt_w=t_x=t_zt_x,\, t_yt_z=t_y=t_wt_y, \,
    t_{w}^2=1=t_{z}^2.
  \]
  The first two of the relations show that $t_z=t_xt_y-1, t_{w}=t_yt_x-1$. We
  can then rewrite the other relations in terms of only $t_x$ and $t_y$. It
  turns out that $J_1$ is generated by
  $t_x$ and $t_y$ subject only to the following two relations: 
  \[
    t_xt_yt_x=2t_x,\, t_yt_xt_y=2t_y.
  \]
  Finally, via the change of variables $X:={t_x}/{2}, Y:=t_y$, we see that
  \[
    J_1=\langle X, Y\rangle/ \langle XYX=X, YXY=Y\rangle.
  \]
  A simple presentation like this is helpful for studying representations of
  $J_1$ and hence $J_2, J_3$ and $J_C$.
\end{example}

\subsection{Fusion $J_s$}
\label{sec:fusion J}
In this subsection, we describe all fusion rings appearing in the form $J_s$ from a Coxeter system.
Recall from Definition \ref{fusion def} that a fusion ring is a unital based
ring of finite rank, so the algebra $J_s$ is a fusion
ring if and only if $\Gamma_{s}\cap\Gamma_{s}\inverse$ is finite. It is 
easy to describe when this happens in terms of Coxeter diagrams.

\begin{proposition}
  \label{graph theory}
  Let $(W,S)$ be an irreducible Coxeter system with Coxeter diagram $G$. Then the following are
  equivalent.
  \begin{enumerate}
    \item $\Gamma_s\cap\Gamma_s^{-1}$ is finite for some $s\in S$; 
      \item  $\Gamma_s\cap\Gamma_s^{-1}$ is finite for all $s\in S$;
        \item $G$ is a tree, no edge of $G$ has weight $\infty$, and at most one
  edge of $G$ has weight greater than 3.
  \end{enumerate}   \end{proposition}

\begin{proof}
  Since $(W,S)$ is irreducible, $G$ is connected. The condition that $G$ is a
  tree is then equivalent to the condition that $G$ contains no cycle. Let
  $D$ be the subregular graph of $(W,S)$, and recall that for each $s\in S$,
  the set $\Gamma_s\cap\Gamma_{s}^{-1}$ correspond bijectively to the $s$-walks
  on $D$. The desired equivalences now follow from a straightforward graph
  theoretic argument, and we omit the details. 
\qed\end{proof}

We can now deduce Theorem
\ref{fusion J}. 
\begin{thm1.3}
  Let $(W,S)$ be a Coxeter system such that $J_s$ is a fusion
  ring for some $s\in S$. Then there exists a dihedral Coxeter system
  $(W',S')$ such that $J_s\cong
  J_{s'}$ as based rings for both $s'\in S'$.  
\end{thm1.3}
\begin{proof}
  For each $n\in \Z_{\ge 3}$, let $(W_n,S')$ be the dihedral system with
  $S'={s',t'}$ and $m(s',t')=n$, and
  let $J_{s'}^{(n)}$ be the ring $J_{s'}$ arising from $(W_n, S')$.

  Let $G$ be the Coxeter diagram of
  $(W,S)$, and suppose $J_s$ is a fusion ring for some $s\in S$. Then
  $\Gamma_s\cap\Gamma_s\inverse$ is finite, hence by Proposition \ref{graph
  theory}, either $G$ is a tree and $(W,S)$ is simply-laced, or
  $G$ is a tree and there exists a unique pair $a,b\in S$ such that
  $m(a,b)>3$. 
  
  In the first case where $(W,S)$ is simply-laced, $J_s$ is isomorphic to the group algebra of $\Pi_s(G)$ by Theorem
  \ref{simply-laced}, and the group is
  trivial since $G$ is a tree, therefore $J_s\cong J_{s'}^{(3)}$. In the second
  case, let $m(a,b)=M$. By the description of
  $G$, there must be a walk $P$ in $G$ from
  $s$ to either $a$ or $b$ where all the edges in the walk have weight 3,
  so Part (4) of Proposition \ref{odd walk} implies that $J_s$ is isomorphic to
  either $J_a$ or $J_b$ as a based ring. We claim that $\Gamma_a\cap \Gamma_{a}^{-1}$ contains exactly
  the elements $a, aba,\cdots, ab\cdots a$ where the reduced words alternate in
  $a,b$ and contain fewer than $M$ letters, so that $J_a\cong J_{s'}^{(M)}$.
  Similarly, $J_b\cong J_{s'}^{(M)}$, therefore $J_s\cong J_{s'}^{(M)}$.

  It remains to prove the claim. It is clear once we note that for each
  $x\in
  \Gamma_a\cap\Gamma_{a}^{-1}$, any spur in the walk $P_x$ from
  Proposition \ref{simply-laced bijection} must involve only $a$ and $b$.
\qed\end{proof}

\begin{corollary}
  \label{fusion verlinde}
  Let $(W,S)$ be a Coxeter system such that $J_s$ is a fusion ring for some
  $s\in S$. Then $J_s\cong \mathrm{Ver}_n^{\text odd}$ for all $s\in S$, where
  $n$ is the highest edge weight in the Coxeter diagram of $(W,S)$.
\end{corollary}
\begin{proof}
  This is immediate from Proposition \ref{dihedral verlinde} and the proof of Theorem
  \ref{fusion J}.
\qed\end{proof}

\section{Free fusion rings} 
\label{sec:free fusion rings}
We focus on certain Coxeter systems $(W,S)$ whose Coxeter
diagrams involve edges of weight $\infty$ in this section. We show that for
suitable choices of $s\in S$, $J_s$ is isomorphic to a \emph{free fusion ring}.

\subsection{Background}
\label{sec:background}
Free fusion rings are defined as follows.
\begin{definition}[\cite{Raum}]
  \label{ffr}
  A \emph{fusion set} is a set $A$ equipped with an {involution}
  \,$\bar{}: A\ra A$ and a \emph{fusion} map $\circ: A \times A \ra A\cup
  \{\varnothing\}$. Given any fusion set $(A,\,\bar{}\,,\circ)$, we extend the operations
  $\,\bar{}\,$ and $\circ$ to the free monoid $\langle A \rangle$ as follows:
  \[
    \overline{a_1\cdots a_k}=\bar a_k\cdots \bar a_1,
  \]
  \[ (a_1\cdots a_k)\circ
    (b_1\cdots b_l)=a_1\cdots a_{k-1}(a_k\circ b_1)b_2\cdots b_l,
  \] 
  where the right side of the last equation is taken to be $\varnothing$  whenever $k=0, l=0$ or
  $a_k\circ b_1=\varnothing$.
  We then define the \emph{free fusion ring} associated with the fusion set
  $(A,\,\bar{}\,,\circ)$ to be the free abelian group $R=\Z\langle
  A \rangle$ with multiplication
  $\cdot$ given by
  \begin{equation}\label{eq:ffr}
    v\cdot w=\sum_{v=xy,w=\bar yz} xz+x\circ z
  \end{equation}
  for all $v, w\in \ip{A}$, where $xz$ means the juxtaposition of $x$ and $z$.
\end{definition}
\noindent It is known that $\cdot$ is associative (see \cite{Raum}). It is
also easy to check that $R$ is a unital based ring with basis $\ip{A}$, with unit given by the empty word, and
with its anti-involution ${}^*:\ip{A}\ra\ip{A}$ given by the map $\,\bar{}\,$.
We should mention that while we have already defined fusion rings to be unital based rings of finite rank in
Definition \ref{fusion def}, the new term ``free fusion rings'' here does not
mean fusion rings with an additional property of being ``free'' in some sense.
In fact, free fusions rings are never fusion rings in the sense of Definition
\ref{fusion def} because they fail to be of finite rank.

Free fusion rings were introduced in \cite{Banica} to capture the
tensor rules in certain semisimple tensor categories arising from the theory
of operator algebras. More specifically,
the categories are categories of representations of \emph{compact quantum
groups}, and their Grothendieck rings
fit the axiomatization of free fusion rings in Definition \ref{ffr}. In
\cite{Freslon-1}, A. Freslon classified all free
fusion rings arising as the Grothendieck rings of compact quantum groups in terms of their underlying
fusion sets. Furthermore, while a free fusion ring may appear as the
Grothendieck ring of multiple non-isomorphic compact quantum groups,
Freslon described a canonical way to associate a \emph{partition
quantum group}---a special type of compact quantum group---to any free fusion
ring arising from a compact quantum group. These special quantum groups
correspond via a type
of Schur-Weyl duality to \emph{categories of non-crossing partitions}, which
can in turn be used to study the representations of the quantum
groups.

All the free fusion rings appearing as $J_s$ in our
examples fit in the classification of \cite{Freslon-1}. In each example, we will identify the associated
partition quantum group  $\mathbb{G}$. The fact that $J_s$ is connected to
$\mathbb{G}$ is intriguing, and it
would be interesting to see how the categorification of $J_s$ arising from
Soergel bimodules connects to the representations of $\mathbb{G}$ on the
categorical level.

\subsection{Example 1: ${O_N^+}$}
\label{sec:example1}
One of the simplest fusion sets is the singleton set $A=\{a\}$ with identity
as its involution and with fusion map $a\circ a=\varnothing$. The associated free fusion
ring is $R=\oplus_{n\in \Z_{\ge 0}} \Z a^n$, where
\[
  a^k\cdot a^l=a^{k+l}+a^{k+l-2}+\cdots +a^{\abs{k-l}}
\]
by Equation \ref{eq:ffr}. The partition quantum group associated to
$R$ is the \emph{free orthogonal quantum group} $O_N^+$, and its
corresponding category of partitions is that of all \emph{noncrossing
 pairings} (see \cite{orthogonal}).

Let $(W,S)$ be the infinite dihedral system with $S=\{1,2\}$ and
$W=I_2(\infty)$, the infinite dihedral group. We claim that $J_1$ is
isomorphic to $R$ as based rings. To see this, recall that $J_s$ is the $\Z$-span of basis
elements $t_{1_n}$, where $n$ is odd and $1_n=121\cdots 1$ alternates in $1,2$
and has length $n$. For $m=2k+1$ and $n=2l+1$ for some $k,l\ge 1$, the
truncated Clebsch--Gordan rule implies that \[ t_{1_m}\cdot
t_{1_n}=t_{1_{2k+1}}t_{1_{2l+1}}=t_{1_{2(k+l)+1}}+t_{1_{2(k+l-1)+1}}+\cdots
+t_{1_{2\abs{k-l}+1}}.  \] It follows that $R\cong J_1$ as based rings via the
unique $\Z$-module map with $a^k\mapsto t_{1_{2k+1}}$ for all $k\in \Z_{\ge
0}$. Similarly, $R\cong J_2$ as based rings. 

\subsection{Example 2: ${U_N^+}$}
\label{sec:example2}
Consider the free fusion ring $R$ arising from the fusion
set $A=\{a,b\}$ with $\bar a=b$ and $a\circ a=a\circ b=b \circ a=b \circ
a=\varnothing$.  The partition quantum group associated to $R$ is the \emph{free
unitary quantum group} $U_N^+$.  In the language of \cite{Freslon-1}, this
quantum group corresponds to the category of \emph{$\mathcal{A}$-colored}
noncrossing partitions where $\mathcal{A}$ is a \emph{color set} containing two
colors \emph{inverse} to each other. 

Consider the Coxeter system $(W,S)$ with the following Coxeter diagram
$G$.
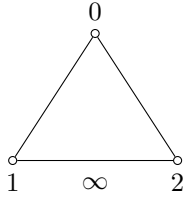
\begin{figure}[h!]
  \begin{center}
    \begin{tikzpicture}
      \node (4) {}; 
      \node[main node] (5) [above right = 0.5cm and 1.5cm of 4] {};
      \node (55) [above = 0cm of 5] {0}; 
      \node[main node] (6) [below left = 1.6cm and 1cm of 5] {};
      \node (66) [below = 0cm of 6] {1}; 
      \node[main node] (7) [below right = 1.6cm and 1cm of 5] {};
      \node (77) [below = 0cm of 7] {2}; 
      \path[draw]
      (5) edge node [left] {} (6)
      (5) edge node [right] {} (7)
      (6) edge node [below = 0.1cm] {$\infty$} (7);
    \end{tikzpicture}
  \end{center}
  \caption{The Coxeter diagram $G$ of $(W,S)$}
\end{figure}

\begin{theorem}
  \label{unitary}
  We have an isomorphism $R\cong J_0$ of based rings.
\end{theorem}

Our strategy to prove Theorem \ref{unitary} is to describe a bijection between the free
monoid $\ip{A}$ and the set
$\Gamma_0\cap \Gamma_0\inverse$, use it to define a $\Z$-module isomorphism
from $R$ to $J_0$, then show that it is an isomorphism of based rings. 

Let $x\in
\Gamma_0\cap\Gamma_0\inverse$, and let $P_x$ be the corresponding path on
$G$ as in Proposition \ref{simply-laced bijection}. Read $P_x$ from left to
right, write down an ``$a$'' every time an edge in the walk goes from $1$
to 2, a ``$b$'' every time an edge goes from $2$ to 1, and write down
nothing otherwise. Call the resulting word $w_x$. Then the map $\varphi: \Gamma_0\cap\Gamma_0\inverse\ra\ip{A}, x\mapsto w_x$ is
is a bijection. The reason is that the walks $\{P_x:x\in \Gamma_0\cap\Gamma_0\inverse\}$
are exactly the walks on $G$ without any spur involving $0$, and from each word in
$\ip{A}$ we may recover a unique such path. For example, suppose $w_x=abaa$,
then $P_x$ must be $(0,1,2,1,2,0,1,2,0)$ and hence $x=012120120$.

We can now prove Theorem \ref{unitary}. We present an inductive proof
that can be easily adapted to prove Theorem \ref{amalgamate} later.

\begin{proof6.1}
  Let $\phi:\ip{A}\ra \Gamma_0\cap\Gamma_0\inverse$ be the inverse of
  $\varphi$, and let $\Phi: R\ra J_0$ be the $\Z$-module homomorphism defined by 
  \[
    \Phi(w) = t_{\phi(w)}.
  \]
  Since $\phi$ is a bijection, this is an isomorphism of $\Z$-modules. We will
  show that $\Phi$ is an algebra isomorphism by showing that 
  \begin{equation}
    \label{eq:check Phi}
    \Phi(v)\Phi(w)=\Phi(v\cdot w)
  \end{equation}
  for all $v,w\in \ip{A}$. Note that this is true if $v$ or $w$ is empty, since then
  $t_v=t_0$ or $t_w=t_0$, which is the identity of $J_0$. 

  Now, assume neither $v$ nor $w$ is empty. We
  prove Equation \eqref{eq:check Phi} by induction on the {length}
  $l(v)$ of $v$, i.e., on the number of letters in $v$.
  For the base case, suppose $l(v)=1$ so that $v=a$ or $v=b$. If
  $v=a$, then $\phi(a)=0120$. There are two cases: 
  \begin{enumerate}[leftmargin=2em]
    \item Case 1: $w$ starts with $a$.  \\
      Then $\phi(w)$ has
      the form $\phi(w)=012\cdots$, so
      \[
        \Phi(v)\Phi(w)=t_{0120}t_{012\cdots}=t_{0120 *
        012\cdots}=t_{012012\cdots}=t_{\phi(aw)}
      \]
      by Proposition \ref{d2}. Meanwhile,
      since $\bar a\neq a$ and $a\circ a=\varnothing$ in $A$, 
      \[ 
        v\cdot w=aw
      \]
      in $R$,
      therefore 
      $
      \Phi(v\cdot w)=t_{\phi(aw)}
      $
      as well. Equation \eqref{eq:check Phi} follows.
    \item Case 2: $w$ starts with $b$. \\
      In this case, suppose the longest alternating subword $bab\cdots$ appearing in the
      beginning of $w$ has length $k$, and and write $w=bw'$.
      Then $\phi(w)$ takes the form $\phi(w)=0212\cdots$; its first
      dihedral segment is $02$ and its second dihedral segment is
      $(2,1)_{k+1}$, so that $\phi(w)=02*(2,1)_{k+1}*x$ where
      $x$ is the glued product of all the remaining dihedral segments. Direct
    computation using Theorem \ref{dihedral factorization} and Propositions
      \ref{d2} and \ref{d3} then yields
      \begin{eqnarray*}
        \Phi(v)\Phi(w)
              &=& t_{01}[t_{(1,2)_{k+2}}+t_{(1,2)_k}]t_{x}\\
        &=& t_{01*(1,2)_{k+2}*x}+t_{01*(1,2)_{k}*x}\\
        &=& t_{\phi(w)}+t_{\phi(w')}.
      \end{eqnarray*}
      Meanwhile, since $\bar a=b$ and $a\circ b=\varnothing$ in $A$, 
      \[
        v\cdot w=a\cdot bab\cdots = abab\cdots + ab\cdots=w+w' 
      \]
      in $R$, therefore $\Phi(v\cdot w)=t_{\phi(w)}+t_{\phi(w')}$ as well.
      Equation \eqref{eq:check Phi} follows.
  \end{enumerate}
  The proof for the case $l(v)=1$ and $v=b$ is similar.

  For the inductive step of our proof, assume Equation \eqref{eq:check Phi}
  holds whenever $v$ is nonempty and $l(v)< L$ for some $L\in \N$,
  and suppose $l(v)=L$. Let $\alpha\in A$ be the first
  letter of $v$, and write $v=\alpha v'$. Then $l(v')<L$, and by \eqref{eq:ffr},
  \[
    a\cdot v'=v+\sum_{u\in U} u
  \]
  where $U$ is a subset of $\ip{A}$ where all words are of length smaller than
  $L$. Using the inductive hypothesis on $\alpha, v',u$ and the
  $\Z$-linearity of $\Phi$, we have
  \begin{eqnarray*}
    \Phi(v)\Phi(w)&=& \Phi\left(\alpha \cdot v'- \sum_{u\in U} u\right)\Phi(w)\\
    &=& \Phi(\alpha)\Phi(v')\Phi(w)-\sum_{u\in U}\Phi(u)\Phi(w)\\
    &=& \Phi(\alpha)\Phi(v'\cdot w)-\Phi\left(\sum_{u\in U}u\cdot w\right).\\
  \end{eqnarray*} 
  Here, the element $v'\cdot w$ may be a linear combination of multiple words in
  $R$, but applying the inductive hypothesis on $\alpha$ still yields 
  \[
    \Phi(\alpha)\Phi(v'\cdot w)=\Phi(\alpha\cdot(v'\cdot w))
  \]
  by the $\Z$-linearity of $\Phi$ and $\cdot$. Consequently,
  \begin{eqnarray*}
    \Phi(v)\Phi(w)    &=& \Phi(\alpha\cdot (v'\cdot w))-\Phi\left(\sum_{u\in U} u\cdot w\right)\\
    &=& \Phi\left(  (\alpha\cdot v')\cdot w- \sum_{u\in U}u\cdot w\right)\\
    &=& \Phi\left(  \left[(\alpha\cdot v')- \sum_{u\in U}u\right]\cdot
    w\right)\\
    &=& \Phi(v\cdot w).
  \end{eqnarray*}
  by the associativity of $\cdot$ and the $\Z$-linearity of $\Phi$ and $\cdot$.
  This completes the proof that $\Phi$ is an algebra isomorphism. 
  
  The fact that
  $\Phi$ is in addition an isomorphism of based rings is straightforward to check.
  In particular, observe that $\phi(\bar w)=\phi(w)^{-1}$ so that  $\Phi(\bar
  w)=t_{\phi(\bar w)}=t_{\phi(w)\inverse}=(\Phi(w))^*$, therefore $\Phi$ is
  compatible with the respective
  involutions in $R$ and $J_0$. We omit the details of the other necessary
  verifications.
\qed\end{proof6.1}

\subsection{Example 3: ${Z_N^+(\{e\},n-1)}$}
\label{sec:example3} Let $[n]=\{1,2,\cdots,n\}$ for each $n\in Z_{\ge 1}$. In this subsection, we consider an infinite family of fusion rings $\{R_n: n\in
  \Z_{\ge 2}\}$, where each $R_n$ arises from the
  fusion set
  \[
    A_n=\{e_{ij}: i,j\in [n]\}
  \]
  with $\bar e_{ij}=e_{ji}$ for all $i,j\in [n]$ and 
  \[
    e_{ij}\circ e_{kl}=
    \begin{cases}
      e_{il} &\quad\text{if}\quad j=k\\
      \varnothing &\quad\text{if}\quad j\neq k
    \end{cases}
  \]
  for all $i,j,k,l\in [n]$. We may think of the fusion set as the usual matrix
  units for $n\times n$ matrices and think of the fusion map as an analog of
  matrix multiplication, with the fusion product being $\varnothing$ whenever the matrix
  product is 0. In the notation of \cite{Freslon-1}, the partition quantum group
  corresponding to $R_n$ is denoted by $Z_N^+(\{e\},n-1)$, which is the
  \emph{amalgamated free product} of $(n-1)$ copies of $\tilde H_N^+$ amalgamated
  along $S_N^+$, where $S_N^+$ stands for the \emph{free symmetric group},
  $H_N^+$ stands for the \emph{free
  hyperoctohedral group}, and $\tilde H_N^+$ stands for the \emph{free
  complexification} of $H_N^+$. In
  particular, $R_2=\tilde H_N^+$.  

  For $n\in \Z_{\ge 2}$, let $(W_n,S_n)$ be the Coxeter system where
  $S_n=[n]$, $m(0,i)=\infty$ for all $i\in [n]$, $m(i,i+1)=3$ for
  all $i\in [n-1]$, and $m(i,j)=2$ otherwise. The Coxeter diagrams $G_n$ of
  $(W_n, S_n)$ are shown in Figure \ref{fig:Rn}, where the thick edges have weight
      $\infty$ and the remaining edges have
    weight 3.

%\vspace{-0.5em}
  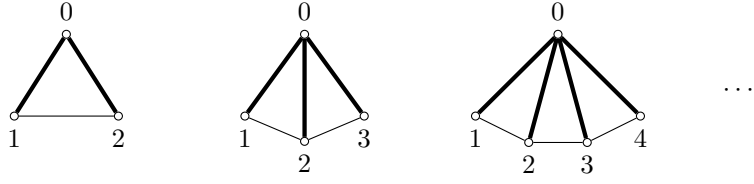
\begin{figure}[h!]
    \begin{centering}
      \begin{tikzpicture}
        \node (4) {}; 
        \node[main node] (5) [above right = 0.5cm and 1.5cm of 4] {};
        \node (55) [above = 0cm of 5] {0}; 
        \node[main node] (6) [below left = 1cm and 0.6cm of 5] {};
        \node (66) [below = 0cm of 6] {1}; 
        \node[main node] (7) [below right = 1cm and 0.6cm of 5] {};
        \node (77) [below = 0cm of 7] {2}; 

        \node (8) [below right = 0.5cm and 0.8cm of 5]
        {};

        \node (9)[main node] [above right = 0.5cm and 2cm of 8]
        {};
        \node (99) [above = 0cm of 9] {0}; 
        \node[main node] (10) [below left = 1cm and 0.7cm of 9]
        {};
        \node (1010) [below = 0cm of 10] {1}; 
        \node[main node] (11) [below =1.3cm of 9] {};
        \node (1111) [below = 0cm of 11] {2}; 
        \node[main node] (12) [below right = 1cm and 0.7cm of 9]
        {};
        \node (1212) [below = 0cm of 12] {3}; 

        \node (13) [below right = 0.5cm and 1cm of 9]
        {};

        \node (14)[main node] [above right = 0.5cm and 2cm of 13]
        {};
        \node (1414) [above = 0cm of 14] {0}; 
        \node[main node] (15) [below left = 1cm and 1cm of 14]
        {};
        \node (1515) [below = 0cm of 15] {1}; 
        \node[main node] (16) [below left = 1.35cm and 0.3cm of 14] {};
        \node (1616) [below = 0cm of 16] {2}; 
        \node[main node] (17) [below right = 1.35cm and 0.3cm of 14] {};
        \node (1717) [below = 0cm of 17] {3}; 
        \node[main node] (18) [below right = 1cm and 1cm of 14]
        {};
        \node (1818) [below = 0cm of 18] {4}; 

        \node (19) [below right = 0.5cm and 2cm of 14] {$\cdots$};
        \path[draw,ultra thick]
        (5) edge node [left] {} (6)
        (5) edge node [right] {} (7)
        (9) edge node [left] {} (10)
        (9) edge node {} (11)
        (9) edge node [right] {} (12)
        (14) edge node {} (15)
        (14) edge node {} (16)
        (14) edge node {} (17)
        (14) edge node {} (18);
        \path[draw]
        (6) edge node [below] {} (7)
        (10) edge node [below] {} (11)
        (11) edge node [below] {} (12)
        (15) edge node [below] {} (16)
        (16) edge node [below] {} (17)
        (17) edge node [below] {} (18);
      \end{tikzpicture}
    \end{centering}
    \caption{The Coxeter diagrams of $(W_n,S_n)$}
    \label{fig:Rn}
  \end{figure}
  Let $J_0^{(n)}$ denote the subring $J_0$ of the subregular $J$-ring of $(W_n,S_n)$.
  \begin{theorem}
    \label{amalgamate}
    For each $n\in \Z_{\ge 2}$,  $R_n\cong J_0^{(n)}$ as
    based rings.
  \end{theorem}
  For each $n\ge 2$, our strategy to prove the isomophism $R_n\cong J_0^{(n)}$ is
  similar to the strategy for Theorem \ref{unitary}. That is, we will first describe a
  bijection $\phi: \ip{A_n}\ra \Gamma_0\cap\Gamma_0\inverse$, then show that the
  $\Z$-module map $\Phi: R_n\ra J_0^{0}$ given by $\Phi(w)=t_{\phi(w)}$ is an
  isomorphism of based rings.

  To describe $\phi$, 
  note that for $i,j\in [n]$, there is a unique shortest walk $P_{ij}$ from
  $i$ to $j$ on the ``bottom part'' of $G_n$, i.e., on the subgraph of $G_n$
  induced by the vertex subset $[n]$. We $\phi(e_{ij})$ to be the element in
  $\Gamma_{0}\cap\Gamma_0\inverse$ corresponding to the walk
  on
  $G$ that starts from 0, travels to $i$ along the edge $\{0,i\}$, traverses to
  $j$ along the path $P_{ij}$, then returns to $0$ along the edge $\{0,j\}$. For example, when $n= 4$,
  $\phi(e_{24})=02340, \phi(e_{43})=0430, \phi(e_{44})=040$. 
  Next, for any word $w$ in $\ip{A_n}$, we define $\phi(w)$ to be the glued product
  of the $\phi$-images of its letters. For example,
  $\phi(e_{24}e_{43}e_{44}e_{44})=023404304040$.  By considering the walks
  $\{P_x:x\in \Gamma_0\cap\Gamma_0\inverse\}$, it is easy to see that
  $\phi:\ip{A}\ra \Gamma_0\cap\Gamma_0\inverse$ is a bijection.

  Before we prove Theorem \ref{amalgamate}, let us record one useful lemma: 
  \begin{lemma}
    \label{path mult}
    Let $x_{ij}=i\cdots j$ be the element in $C$ corresponding to
    the walk $P_{ij}$ for all $i,j\in [n]$. Then
    $t_{x_{ij}}t_{x_{jk}}=t_{x_{ik}}$ for all $i,j,k\in [n]$.
  \end{lemma}

  \begin{proof}
    We can verify this by considering all possible relationships between
    $i,j,k$ and directly computing $t_{x_{ij}}t_{x_{jk}}$ in each case as
    discussed in Section \ref{sec:product computation}. For example, in the
    case $i<j<k$, the claim follows from Proposition \ref{d2}.
  \qed\end{proof}

  \begin{proof6.2}
    Let $n\ge 2$, and let $\phi$ and $\Phi$ be as above. As in the proof of
    Theorem \ref{unitary}, we show that $\Phi$ is an algebra isomorphism by
    checking that 
    \begin{equation}
      \Phi(v)\Phi(w)=\Phi(v\cdot w)
      \label{eq:check Phi again}
    \end{equation}
    for all $v,w\in \ip{A_n}$. Once again, we may assume that both $v$ and $w$
    are non-empty and use induction on the length $l(v)$ of $v$. The inductive
    step of the proof will be identical with the one for Theorem
  \ref{unitary}.
    For the base case where $l(v)=1$, suppose $v=e_{ij}$ for some $i,j\in [n]$.
    There are two cases.  
    \begin{enumerate}[leftmargin=2em]
      \item Case 1: $w$ starts with a letter $e_{j'k}$ where $j'\neq j$.  \\
        Then $\phi(v)$ and $\phi(w)$ take the form $\phi(v)=\cdots
        j0,\phi(w)=0j'\cdots$, so 
        \[
          \Phi(v)\Phi(w)=t_{\cdots j0}t_{0j'\cdots}=t_{\cdots j0 *
          0j'\cdots}=t_{\phi(e_{ij})*\phi(w)}=t_{\phi(e_{ij}w)}
        \]
        by Proposition \ref{d2}. Meanwhile,
        since $\bar e_{ij}\neq e_{j'k}$ and $e_{ij}\circ e_{j'k}=\emptyset$ in
        $A_n$, 
        \[ 
          v\cdot w=e_{ij}w
        \]
        in $R$,
        therefore 
        $
        \Phi(v\cdot w)=t_{\phi(e_{ij}w)}
        $
        as well. Equation \eqref{eq:check Phi again} follows.
      \item Case 2: $w$ starts with $e_{jk}$ for some $k\in [n]$. \\
        Write $w=e_{jk}w'$. We need to consider four subcases, according to how they affect
        the dihedral segments of $\phi(v)$ and $\phi(w)$.
        \begin{enumerate}
          \item $i=j=k$.
            Then $v=e_{jj}$, $\phi(v)=0j0=(0,j)_3$ (see Definition
            \ref{subregular graph}), and $w$ starts with
            $e_{jj}\cdots$, hence $\phi(w)$ starts with $0j0\cdots$. Suppose the
            first dihedral segment of $\phi(w)$ is $(0,j)_L$, and write
            $\phi(w)=(0,j)_L* x$. Then Theorem
            \ref{dihedral factorization} and Propositions \ref{d2} and
            \ref{d3} yield
            \begin{eqnarray*}
              \Phi(v)\Phi(w)&=& t_{(0,j)_3}t_{(0,j)_L}t_{x}\\
              &=&
              t_{(0,j)_{L+2}*{x}}+t_{(0,j)_{L}*{x}}+t_{(0,j)_{L-2}*{x}}\\&=&
              t_{\phi(e_{jj}w)}+t_{\phi(w)}+t_{\phi(w')},
            \end{eqnarray*}
            while
            \[
              v\cdot w=e_{jj}\cdot e_{jj}w'=e_{jj}e_{jj}w' + e_{jj}w'+{w'}=e_{jj}w+w+w'
            \]
            since $\bar e_{jj}=e_{jj}$ and $e_{jj}\circ e_{jj}=e_{jj}$. It follows that
            Equation \eqref{eq:check Phi again} holds. 
          \item $i=j$, but $j\neq k$. In this case, $v=e_{jj}, \phi(v)=(0,j)_3$ as in
            (a), while $\phi(w)=0j*x$ for some reduced word $x$ which starts with $j$ but
            not $j0$. We have
            \[
              \hspace{3em}
              \Phi(v)\Phi(w)=t_{0j0}t_{j0}t_{x}=t_{0j0j*x}+t_{0j*x}=t_{\phi(e_{jj}w)}+t_{\phi(w)},
            \]
            while 
            \[
              v\cdot w=e_{jj}\cdot e_{jk}w'=e_{jj}e_{jk}w'+e_{jk}w'={e_{jj}}w+w
            \]
            since $\bar e_{jj}\neq e_{jk}$ and $e_{jj}\circ e_{jk}=e_{jk}$. This implies
            Equation \eqref{eq:check Phi again}.
          \item $i\neq j$, but $j=k$. In this case, $v=e_{ij}$ and $\phi(v)=y*j0$ for
            some reduced word $y$ which ends in $j$ but not $0j$, and $\phi(w)$ can be written as
            $\phi(w)=(0,j)_{L}*x$ as in (a). We have
            \begin{eqnarray*}
              \Phi(v)\Phi(w)&=& t_{y}t_{j0}t_{(0,j)_{L}}t_{x}\\
              &=&
              t_{y*(j,0)_{L+1}*{x}}+t_{y*(j,0)_{L-1}*{x}}\\
              &=&  t_{\phi(e_{ij}w)}+t_{\phi(e_{ij}w')},
            \end{eqnarray*}
            while 
            \[
              v\cdot w=e_{ij}\cdot e_{jj}w'=e_{ij}w+e_{ij}w'
            \]
            since $\bar e_{ij}\neq e_{jj}$ and $e_{ij}\circ e_{jj}=e_{ij}$. This
            implies Equation \eqref{eq:check Phi again}.
          \item $i\neq j$, and $j\neq k$. In this case, $\phi(v)=0i*x_{ij}*j0$
            (recall the definition of $x_{ij}$ from Lemma \ref{path mult}), and
            $\phi(w)=0j*x_{jk}*x$ for some $x$ which starts with $k0$. We have
            \begin{eqnarray*}
              \Phi(v)\Phi(w)&=& t_{0i}t_{x_{ij}}t_{j0}t_{0j}t_{x_{jk}}t_x\\
              &=&
              t_{0i}t_{x_{ij}}t_{j0j}t_{x_{jk}}t_{x}+t_{0i}t_{x_{ij}}t_{j}t_{x_{jk}}t_{x}\\
              &=& t_{0i*x_{ij}*j0j*x_{jk}*x}+t_{0i}t_{x_{ij}}t_{x_{jk}}t_{x}\\
              &=& t_{\phi(e_{ij}w)}+t_{0i}t_{x_{ik}}t_{x},
            \end{eqnarray*}
            where $t_{x_{ij}}t_{x_{jk}}=t_{x_{ik}}$ by Lemma \ref{path
            mult}. Now, if $i\neq k$, then
            $t_{0i}t_{x_{ik}}t_{x}=t_{0i*x_{ik}*x}=t_{\phi(e_{ik}w')}$, so
            \[
              \Phi(v)\Phi(w)=t_{\phi(e_{ij}w)} + t_{\phi(e_{ik}w')}.
            \] 
            If $i=k$, note that $t_{0i}t_{ik}t_{x}=t_{0k}t_kt_{x}=t_{0k}t_x$. Suppose the first dihedral segment of $x$ is
            $(k,0)_{L'}$ for some $L'\ge 2$, and write $x=(k,0)_{L'}* x'$. Then
            $t_{0k}t_{x}=t_{0k}t_{(k,0)_{L'}}t_{x'}=
            t_{(0,k)_{L'+1}*x'}+t_{(0,k)_{L'-1}*x'}=t_{\phi(e_{kk}w')+\phi(w')}$,
            so
            \[
              \Phi(v)\Phi(w)=t_{\phi(e_{ij}w)}+t_{\phi(e_{ik}w')}+t_{\phi(w')}.
            \]
                        Meanwhile, 
            \[
              \hspace{4.5em} v\cdot w= e_{ij}\cdot e_{jk}w'
              =e_{ij}e_{jk}w'+e_{ik}w'+\delta_{ik}e_{w'}
              =e_{ij}w+e_{ik}w'+\delta_{ik}e_{w'}  
            \]
            because $\bar e_{ij}=e_{ji}$ and $e_{ij}\circ e_{jk}=e_{ik}$,
            therefore Equation \eqref{eq:check Phi again} holds again.
        \end{enumerate}
    \end{enumerate}
    We have proved that $\Phi$ is an algebra isomorphism. As in Theorem
    \ref{unitary}, the fact that $\Phi$ is in
    addition an isomorphism of based rings is easy to check, and we omit
    the details.
  \qed\end{proof6.2}

\begin{acknowledgements} It is my great pleasure to thank Victor Ostrik for numerous helpful suggestions. I am very grateful to
Alexandru Chirvasitu and Amaury Freslon for helpful discussions about free fusion rings
and compact quantum groups. I would also like to acknowledge the
mathematical software {\tt SageMath} (\cite{sagemath}), which was used
extensively in our computations.
\end{acknowledgements}

% bibliography
    %\bibliography{J.bib}
    %\bibliographystyle{alpha}
    %\nocite{*}

%\begin{acknowledgements}
%\end{acknowledgements}

\bibliographystyle{spmpsci}      % mathematics and physical sciences
\bibliography{subregular_Tianyuan_Xu.bib}   % name your BibTeX data base

%% Non-BibTeX users please use
%\begin{thebibliography}{}
%%
%% and use \bibitem to create references. Consult the Instructions
%% for authors for reference list style.

%\end{thebibliography}

\end{document}